\documentclass[11pt,a4paper]{article}
\usepackage[centertags]{amsmath}
\usepackage{amsfonts}
\usepackage{amssymb}
\date{}
\usepackage{amsthm}

\usepackage{caption}
\usepackage{enumerate}
\usepackage{fancybox}
\usepackage{float}

\usepackage[text={6.7in,9.5in},centering]{geometry}

\usepackage{hyperref}

\usepackage{newlfont}
\usepackage{subcaption}
\usepackage{todonotes}
\usepackage{txfonts}
\usepackage{url}
\usepackage{varioref}
\newcommand{\md}{\thinspace \mathrm{d}} 
\newcommand{\RR}{{\mathbb R}}
\theoremstyle{plain}
\newtheorem{theorem}{Theorem}[section]
\newtheorem{corollary}[theorem]{Corollary}
\newtheorem{lemma}[theorem]{Lemma}
\newtheorem{proposition}[theorem]{Proposition}
\theoremstyle{definition}
\newtheorem{definitioin}{Definition}[section]
\theoremstyle{remark}
\newtheorem{remark}{Remark}[section]
\numberwithin{equation}{section}

\newcommand{\absn}[1]{\lvert\thinspace {#1} \thinspace\rvert}
\newcommand{\arrayoptions}[2]{\setlength{\arraycolsep}{#1}\renewcommand{\arraystretch}{#2}}
\newcommand{\defeq}{\coloneqq}
\newcommand{\ess}{\ensuremath{\mathrm{ess}}} \newcommand{\eg}{e.\,g.}
\newcommand{\ie}{i.\,e.}

\floatstyle{ruled} 
\restylefloat{figure}
\restylefloat{table}

\begin{document}

\title{\bf Uniqueness and Stability of Optimizers for a Membrane
  Problem}

\author{Behrouz Emamizadeh\thanks{School of Mathematical Sciences,
    University of Nottingham Ningbo China, 199 Taikang East Road,
    Ningbo, Zhejiang China 315100,
    \href{mailto:Behrouz.Emamizadeh@nottingham.edu.cn}{Behrouz.Emamizadeh@nottingham.edu.cn}}
  \and Amin Farjudian\thanks{School of Computer Science, University of
    Nottingham Ningbo China, 199 Taikang East Road, Ningbo, Zhejiang
    China 315100,
    \href{mailto:Amin.Farjudian@gmail.com}{Amin.Farjudian@gmail.com}}\,
  \thanks{Corresponding author} \and {Yichen Liu}\thanks{Beijing
    International Center for Mathematical Research, Peking University,
    China,
    \href{mailto:yichen.liu07@yahoo.com}{yichen.liu07@yahoo.com}} \and
  Monica Marras\thanks{Dipartimento di Matematica e Informatica,
    Universit\`a di Cagliari, viale Merello 92, 09123 Cagliari, Italy,
    \href{mailto:mmarras@unica.it}{mmarras@unica.it}}}

\maketitle

\begin{abstract}

  We investigate a PDE-constrained optimization problem, with an
  intuitive interpretation in terms of the design of robust membranes
  made out of an arbitrary number of different materials. We prove
  existence and uniqueness of solutions for general smooth bounded
  domains, and derive a symmetry result for radial ones. We strengthen
  our analysis by proving that, for this particular problem, there are
  no non-global local optima. When the membrane is made out of two
  materials, the problem reduces to a shape optimization problem. We
  lay the preliminary foundation for computable analysis of this type
  of problem by proving stability of solutions with respect to some of
  the parameters involved.  
\end{abstract}


\begin{description}
\item[Key Words:] Optimization, Stability, Radial symmetry, Boundary
  value problem, Rearrangements of functions.
\item[Mathematics Subject Classification:] 65K10, 74H55, 35J20, 35J25.
\end{description}



\section{Introduction}

Consider the boundary value
problem:
\begin{equation}
  \label{eq:bvp_main}
  \left\{
    \arrayoptions{1ex}{1.2}
    \begin{array}{ll}
      -\Delta u+g(x)u=f(x), &\text{in}\;D, \\
      u=0, &\text{on}\;\partial D,
    \end{array}
  \right.
\end{equation}
in which $D\subseteq \RR^N$ is a smooth domain, $N \in \{ 2,3 \}$, $g$
is a non-negative function in $L^\infty(D)$, and $f$ is a non-negative
function in $L^2(D)$.

When the range of the function $g$ is a finite set, say, $\{\alpha_1,
\alpha_2, \ldots, \alpha_n\}$, this equation may be interpreted in a
very intuitive way. Indeed, the boundary value
problem~(\ref{eq:bvp_main}) models an elastic membrane, constructed
out of $n$ different materials, fixed around the boundary, and subject
to a vertical force $f(x)$ at each point $x$. The solution $u$ denotes
the displacement of the membrane from the rest position.

Let us assume that we have been given the $n$ constituent materials,
together with the force $f$ and the geometry of the domain, and our
task is to construct a \emph{robust membrane} out of the given
materials. In this paper, we demonstrate how this may be
achieved. More formally, we associate the following {\em energy}
functional with the boundary value problem~(\ref{eq:bvp_main}):
\begin{equation}
\label{eq:energy_main}
\Phi(g) \defeq \int_Dfu_g \md x=\int_D|\nabla u_g|^2 \md x+\int_Dgu_g^2 \md x,
\end{equation}
in which, $u_g$ is the unique solution of (\ref{eq:bvp_main}). At an
intuitive level, this energy functional is mean to measure the
\emph{vulnerability} of the membrane. It should be straightforward to
verify that the following identity follows from the variational
formulation of $u_g$:
\begin{equation}
  \label{eq:variational_formulation}
  \Phi(g)=\sup_{v\in H^1_0(D)}\left\{2\int_Dfv \md x-\int_D(|\nabla
    v|^2+gv^2) \md x\right\}.
\end{equation}
We assume that the information about the constituent materials is
provided in a given function $g_0$ which satisfies $0\le g_0\le 1$,
and which is is not identically zero. We let
$\cal{R}\equiv\cal{R}(g_0)$ denote the \emph{rearrangement} class
generated by $g_0$ (Definition~\ref{eq:rearrangeclass}
\vpageref[below]{eq:rearrangeclass}). To obtain a robust membrane, we
need to obtain the arrangement of $g_0$ with the least vulnerability,
{\ie}, we need to solve the following minimization problem:
\begin{equation}
\label{eq:main_problem}
\inf_{g\in\cal{R}}\Phi(g).
\end{equation}

\begin{remark}
  The minimization problem~(\ref{eq:main_problem}) is of interest from
  a pure mathematical perspective as well. Indeed, the maximum
  principle ensures that $u_g$, the solution of~(\ref{eq:bvp_main}),
  is positive. Hence, the integral $\int_D fu_g  \md x$ is the
  $L^1(\mu)$-norm of $u_g$, with $d\mu$ being the measure which is
  absolutely continuous with respect to the Lebesgue measure $ \md x$,
  having $f$ as its Radon-Nikodym derivative with respect to
  $ \md x$. Minimization of various norms of solutions of partial
  differential equations is a classical topic of interest among
  mathematicians.  
\end{remark}

\subsection{Approach and contributions}

Our approach towards proving the solvability of
(\ref{eq:main_problem}) is based on the well-developed theory of
\emph{rearrangements of
  functions}~\cite{Talenti:Art_Rearranging:2016}. Specifically,
we use the theory developed by
G.~R.~Burton~\cite{Burton:1987,Burton:1989} for optimization over
rearrangement classes. To this end, we first relax the minimization
problem (\ref{eq:main_problem}) by extending the admissible set
$\cal{R}$ to its weak closure $\overline{\cal{R}}$ with respect to
$L^2$-topology. Once the relaxed problem is shown to be solvable, we
will demonstrate how the appropriate restrictions on the force
function $f$ imply that solutions of the relaxed problem are indeed
solutions of the original problem (\ref{eq:main_problem}).

We strengthen our results by proving that the optimization
problem~(\ref{eq:main_problem}) has no non-global local optima, and by
showing that, when $D$ is a ball and $f$ is radial, then the solution
of (\ref{eq:main_problem}) is radial and non-increasing. 

\begin{remark} An appealing aspect of our method is that it can also
  be used when the function $g$ belongs to the larger class $L^p(D)$,
  for $1<p<\infty$, in which case, only minor modifications will be
  required. We prefer, however, to focus on the case $g\in L^\infty
  (D)$, in order to minimize technicalities, and keep the model more
  realistic.
\end{remark}

In the second part of the paper, we discuss some stability
results. These results are of utmost importance in setting up a
framework for computable analysis of problems such as our main
problem~(\ref{eq:main_problem}).

\subsection{Related work}

For any given set $E \subseteq D$, by $\chi_E$ we denote the
characteristic function of $E$, {\ie}, $\chi_E(x)=1$ if $x\in E$, and
$\chi_E(x)=0$ if $x \notin E$. Henrot and
Maillot~\cite{Henrot_Maillot:2001:Optim_Shape_Actuator} have
investigated the special case of the minimization problem
(\ref{eq:main_problem}), in which $g_0=\chi_{E_0}$, for some $E_0
\subseteq D$ with $|E_0|=\alpha$. Under this assumption, one would get
${\cal R}=\{\chi_E: E \subseteq D \wedge |E|=\alpha\}$. In simple
terms, the rearrangement class generated by $g_0$ would be exactly the
set of all characteristic functions of those measurable subsets of $D$
that have the same Lebesgue measure as $E_0$.

Henrot and Maillot~\cite{Henrot_Maillot:2001:Optim_Shape_Actuator}
prove the solvability for this special case, and state the minimality
condition in terms of tangent cones. Since the underlying function
space is $L^\infty(D)$, they are able to derive a convenient
formulation of the tangent cone of an appropriate convex set.

The method employed in \cite{Henrot_Maillot:2001:Optim_Shape_Actuator}
is inadequate for addressing the optimization problem
(\ref{eq:main_problem}) for general generators $g_0$. The theory that
we shall introduce in this paper, however, not only furnishes an
answer to the aforementioned question, but also can be used for a
broader range that includes other design problems.

The second part of the current paper addresses some further issues,
including stability properties of the solutions. This is part of a
broader programme of laying the foundations for robust computable
analysis of rearrangement optimization problems in particular, and
shape optimization problems in general. In this regard, we have
carried out some general stability analyses pertaining to
rearrangement optimization classes, which may be found
in~\cite{Liu_Emamizadeh_Farjudian:Optim_fixed_volume_stability:2016}.

\begin{remark}
  Parts of an earlier draft of this article have appeared in the PhD
  dissertation of one of the
  co-authors~\cite[Sec.~3.3]{Liu:PhD_Thesis:2015}.
\end{remark}

\subsection{Structure of the paper}

The remainder of the paper is structured as follows:

\begin{itemize}
\item Section~\ref{sec:preliminaries} contains preliminary material
  from the theory of rearrangements of functions.

\item In Section~\ref{sec:exist-uniq-optim} we prove existence and
  uniqueness of optimal solutions, and provide a radial symmetry
  result as well. For the minimization problem, we will show that
  there are no non-global local optima. Finally, we provide some
  remarks on the corresponding maximization problem.

\item In Section~\ref{sec:shape-optimization}, we discuss the shape
  optimization variant of the main problem. Specifically, we will
  discuss monotonicity and stability results related to the case where
  the generator is two-valued.

\item In Section~\ref{sec:numerical-simulation}, we provide some
  remarks on the numerical simulation of the optimization problem.
\item In order to avoid breaking the flow of the paper, the lengthy
  proof of Lemma~\ref{lem:basic_props_of_energy_func} (from
  Section~\ref{sec:exist-uniq-optim}) is moved to
  Section~\ref{sec:proof-lemma-phi}.

\item In Section~\ref{sec:concluding-remarks}, we finish the paper
  with some concluding remarks.
\end{itemize}


\section{Preliminaries}
\label{sec:preliminaries}

In this section, we recall some well-known results from the theory of
rearrangements of functions. Henceforth, we denote the $N$-dimensional
Lebesgue measure of a measurable set $E$ by $|E|$. Moreover, for a
Lebesgue measurable function $h:D\to[0,\infty)$ and $\alpha \geq 0$,
we let:
\begin{equation*}
  \lambda_{h}(\alpha)\defeq | \left\{
  x\in D: h(x)\ge\alpha\right\}|.
\end{equation*}

\begin{definitioin}
  \label{def:rearrange}
  Let $g, g_0:D\to[0,\infty)$ be Lebesgue measurable. We say that $g$ is a
  rearrangement of $g_0$ if and only if $\forall \alpha \geq 0: \lambda_{g_0}(\alpha) = \lambda_{g}(\alpha)$.
\end{definitioin}

\begin{definitioin}
  \label{def:inrearrange} For a Lebesgue measurable
  $g:D\to[0,\infty)$, the essentially unique decreasing rearrangement
  $g^\Delta$ is defined on $(0,|D|)$ by $g^\Delta(s) \defeq
  \max\left\{\alpha:\lambda_g(\alpha)\ge s\right\}$. The essentially
  unique increasing rearrangement $g_\Delta$ of $g$ is defined by
  $g_\Delta(s) \defeq g^\Delta(|D|-s)$.
\end{definitioin}

\begin{definitioin}
\label{eq:rearrangeclass}
The set $\cal{R}\equiv\cal{R}(g_0)$, called the rearrangement class generated by $g_0$, is defined as follows
\begin{equation*}
  \cal{R}(g_0) \defeq \left\{g:D\to[0,\infty):g\,\ \text{is a rearrangement of}\,\ g_0\right\}.
\end{equation*}
\end{definitioin}

\begin{definitioin}
  For a function $f:D\to[0,\infty)$, we say that the graph of $f$ has
  no significant flat sections on $D$ if $\forall c \geq 0: \left|\left\{x\in
      D:f(x)=c\right\}\right|=0$.
\end{definitioin}

Henceforth, the support of $g$ will be denoted by
$S(g)\equiv\left\{x\in D:g(x)>0\right\}$, and the reader should
distinguish this definition of support from the usual topological
definition. We use $\overline{\cal{R}}$ to denote the weak closure of
$\cal{R}$ in $L^2(D)$. It is well-known that $\overline{\cal{R}}$ is
convex, and weakly compact in $L^2(D)$.

\begin{lemma}
  \label{lem:rearrange1} Let $\overline{\cal{R}}$ be the weak closure
  of $\cal{R}$ in $L^2(D)$. Then, $\overline{\cal{R}}\subseteq
  L^\infty(D)$, and $\forall g\in \overline{\cal{R}}:
  \left\|g\right\|_\infty\le\left\|g_0\right\|_\infty$.
\end{lemma}

\begin{proof}
  In order to derive a contradiction, we suppose $g\notin
  L^\infty(D)$. Hence, for every positive $M$:
  \begin{equation*}
    |\left\{x\in
      D:g(x)>M\right\}|>0.
  \end{equation*}
  Let us choose $M=\left\|g_0\right\|_\infty$, and set $E \defeq \left\{x\in
    D:g(x)>\left\|g_0\right\|_\infty\right\}$. Since
  $g\in\overline{\cal{R}}$, there exists
  $\left\{g_n\right\}\subseteq\cal{R}$ such that $g_n\rightharpoonup
  g$ in $L^2(D)$. Then, we have:
\begin{equation}
\label{eq21}
\int_{E} g_n \md x=\int_D g_n\chi_{E} \md x\to\int_D g\chi_{E} \md x=\int_{E} g \md x.
\end{equation}
From the definition of $E$ and the fact that $\int_{E}g_n \md x\le
\left\|g_0\right\|_\infty|E|$, in conjunction with (\ref{eq21}), we
deduce:
\begin{equation}
\label{eq22}
\left\|g_0\right\|_\infty|E|<\int_{E}g \md x=\lim_{n\to\infty}\int_E
g_n \md x\le \left\|g_0\right\|_\infty|E|.
\end{equation}
Obviously, (\ref{eq22}) is a contradiction. The above argument implies
that the measure of $E$ is zero.  Hence,
$\left\|g\right\|_\infty\le\left\|g_0\right\|_\infty$. This completes
the proof of the lemma.
\end{proof}

\begin{lemma}
  \label{lem:rearrange3}
  Suppose $\left\{g_n\right\}\subseteq L_+^\infty(D)$, and $g\in
  L^2(D)$. Suppose $g_n\rightharpoonup g$ in $L^2(D)$.  Then, $g$ is
  non-negative a.e. in $D$.
\end{lemma}

\begin{proof} This is an immediate consequence of Mazur's Lemma. Indeed,
  by Mazur's Lemma, there exists a sequence $\left\{v_n\right\}$ in the
  convex hull of the set $\left\{g_n:n\in\mathbb{N}\right\}$ such that
  $v_n\to g$ in $L^2(D)$. Therefore, $v_n\to g$ in measure. Whence,
  there exists a subsequence of $\left\{v_n\right\}$ which converges
  to $g$ a.e. in $D$. This completes the proof.
\end{proof}

The next lemma is easy to prove:

\begin{lemma}
\label{rearrange4}
Suppose that $f:D\to [0,\infty)$ is measurable. Then, for every
measurable subset $E\subseteq D$:
  \begin{equation*}
    \int_Ef \md x\ge\int_0^{|E|}f_\Delta(s) \md s.
  \end{equation*}
\end{lemma}

\begin{lemma}
\label{rearrange2}
For every $g$ in $\overline{\cal{R}}$ we have $|S(g_0)|\le|S(g)|$.
\end{lemma}

\begin{proof}
  In order to derive a contradiction, let us assume that
  $|S(g)|<|S(g_0)|$. Hence,
  $\alpha\equiv\int_0^{|S(g)^c|}g_{0_{\Delta}}  \md x$ is positive.  Since
  $g\in\overline{\cal{R}}$, there exists
  $\left\{g_n\right\}\subseteq\cal{R}$ such that $g_n\rightharpoonup
  g$ in $L^2(D)$. Then, we have:
  \begin{equation}
    \label{eq23}
    \alpha=\int_0^{|S(g)^c|}g_{0_{\Delta}}
     \md x=\int_0^{|S(g)^c|}g_{n_{\Delta}}  \md x \le  \int_{S(g)^c} g_n \md x
     =
    \int_D g_n\chi_{S(g)^c} \md x\to \int_D g\chi_{S(g)^c} \md x=\int_{S(g)^c} g \md x=0,
  \end{equation}
  which is a contradiction. The inequality in (\ref{eq23}) is a
  consequence of Lemma~\ref{rearrange4}.
\end{proof}

We make use of the following lemmata from~\cite{Burton:1987} and~\cite{Burton:1989}.

\begin{lemma}
  \label{lem:rearrange5}
  The following characterization for the weak closure of $\cal{R}$
  holds:
  \begin{equation*}
   \overline{\cal R}
 =
 \left\{g\in L^1(D): \int_Dg \md x = \int_Dg_0 \md x \quad \text{and} \quad \forall s\in
   (0,|D|): \int_0^sg^\Delta  \md t\leq\int_0^sg_0^\Delta  \md t
 \right\} .
  \end{equation*}
\end{lemma}

\begin{proof} See Lemma 2.3 in~\cite{Burton:1989}.
\end{proof}

In line with the established convention
of~\cite{Burton:1987,Burton:1989}, in what follows we often write
`increasing' instead of non-decreasing and `decreasing' instead of
non-increasing.

\begin{lemma}
  \label{lem:zeroflat}
  Suppose that $f:D\to [0,\infty)$ is measurable and has no
  significant flat sections on $D$. Then, there exists an increasing
  function $\psi$ such that $\psi(f)$ is a rearrangement of $g_0$.
  Moreover, there is a decreasing function $\tilde\psi$ such that
  $\tilde\psi(f)$ is a rearrangement of $g_0$.
\end{lemma}

\begin{proof}
  See Lemma 2.9 in~\cite{Burton:1989}.
\end{proof}

\begin{lemma}
  \label{lem:uniquemaximizer}
  Let $f\in L^2(D)$ be a non-negative and non-trivial function ({\ie},
  it is not identically zero), and assume that there is an increasing
  function $\psi$ such that $\psi(f)\in \cal{R}$. Then $\psi(f)$ is
  the unique maximizer of the linear functional $L(h) \defeq
  \int_D fh \, \md x$ relative to $h\in\overline{\cal{R}}.$
\end{lemma}

\begin{proof}
  See Lemma 2.4 in~\cite{Burton:1989}.
\end{proof}

We will also need the following rearrangement result for the Dirichlet
integral (see, {\eg},~\cite{BROTHERS_ZIEMER:1988}). Note that
here $v^*$ denotes the Schwarz symmetrization of $v$ (see,
{\eg},~\cite{Kawohl:1985}):

\begin{lemma}{\ }
  \label{lem:Dirichlet}
  \begin{enumerate}[(i)]
  \item{\label{item:dirich_integ_nabla_vstar_lt_v}} If $v\in
    H^1_0(\RR^N)$ is non-negative, then, $v^*\in H^1_0(\RR^N)$, and the
    following inequality holds:
    \begin{equation}
      \label{eq:PZ}
      \int_{\RR^N}|\nabla v^*|^2\; \md x\leq\int_{\RR^N}|\nabla v|^2\; \md x.
    \end{equation}

  \item{\label{item:dirich_v_trans_vStar}} If $v\in H^1_0(\RR^N)$ is
    non-negative and equality holds in (\ref{eq:PZ}), then, for every
    $0\leq\alpha<M:=\text{ess~sup}~v$, $v^{-1}(\alpha,\infty)$ is a
    translate of the disk ${v^*}^{-1}(\alpha,\infty)$, almost
    everywhere. If, in addition, $\{x\in\RR^N:\;\nabla v=0,\;0<v(x)<M\}$
    has zero measure, then $v$ is a translate of $v^*$.
  \end{enumerate}
\end{lemma}

\section{Optimal solutions}
\label{sec:exist-uniq-optim}

We need to make certain assumptions on the force function $f$ in order
to be able to obtain our main results. Henceforth, $v_f\in H_0^1(D)$
will denote the unique solution of the Poisson boundary value problem:
\begin{equation}
  \label{eq:Poisson_A1}
  \left\{
    \arrayoptions{0.3ex}{1.2}
    \begin{array}{rcll}
      -\Delta v_f & = & f &\quad \text{in }D,\\
      v_f & = & 0 &\quad \text{on } \partial D.
    \end{array}
  \right.
\end{equation}
Here is the main assumption on which our results will hinge:

\begin{center}
  \Ovalbox{\parbox[t]{0.6\textwidth}{
      \begin{center}
        {\bf A1:} \ $v_f\le f$, in $D$.
      \end{center}
}}
\end{center}

A minor problem with this assumption is that its statement involves
the solution to the Poisson boundary value problem
(\ref{eq:Poisson_A1}). It turns out that we can also work with the
following assumption, whose statement involves just the function $f$
and its Laplacian:
\begin{center}
      \begin{center}
        {\bf A2:} \ $f\le-\Delta f$, in $D$.
      \end{center}
\end{center}

\begin{proposition}
\label{prop:A2_implies_A1}
{\bf A2} implies {\bf A1}.
\end{proposition}

\begin{proof}
  Notice that we have:

\begin{equation}
  \label{eq:bvp_A2_implies_A1}
  \left\{
    \arrayoptions{1ex}{1.2}
    \begin{array}{ll}
      -\Delta (v_f-f)=f+\Delta f &\text{in}\;D \\
      v_f-f\le0 &\text{on}\;\partial
      D.
    \end{array}
  \right.
\end{equation}
Since $f+\Delta f$ is non-positive, we can apply the maximum principle
to (\ref{eq:bvp_A2_implies_A1}) to deduce $v_f\le f$.
\end{proof}

As a consequence, \emph{all of the results that will be proved based
  on {\bf A1} will also hold for assumption {\bf A2}}.

\begin{remark}
  Our assumptions are valid, in the sense that there are non-negative
  functions satisfying {\bf A2}, and by implication {\bf A1}. Indeed,
  consider the boundary value problem
  \begin{equation}
    \label{bvp6}
    \left\{
      \arrayoptions{0.3ex}{1.2}
      \begin{array}{rcll}
        -\Delta u-u & = & N & \quad \text{in }D, \\
        u & = & 0 & \quad \text{on } \partial D,
      \end{array}
    \right.
  \end{equation}
  in which $N\in[0,\infty)$. The energy functional associated with
  (\ref{bvp6}) is:
  \begin{equation*}
    I(u)=\frac{1}{2}\int_D|\nabla
    u|^2 \md x-\frac{1}{2}\int_D u^2 \md x-\int_DNu \md x.
  \end{equation*}
  It is clear from the Poincar$\acute{\text{e}}$ inequality that, if
  $D$ is thin, then $I(u)$ will be coercive. So, by an application of
  the direct method of calculus of variations to the functional
  $I(u)$, we infer the existence of a critical point which is a
  solution of (\ref{bvp6}). In order to show that (\ref{bvp6}) has a
  non-negative solution, it suffices to point out that $I(|u|)\le
  I(u)$.
\end{remark}

\subsection{Existence, uniqueness, and optimality condition}

Our assumptions guarantee that the solution $u_g$ of the boundary
value problems (\ref{eq:bvp_main}) has no significant flat sections on
$S(g)$, a fact which will be used in the proof of our main result:

\begin{lemma}
  \label{lem:estimate}
  Suppose that $f$ satisfies assumption {\bf A1}, and $g$ is a
  measurable function such that $0\le g\le 1$. Then, $u_g$ has no
  significant flat sections on $S(g)$.
\end{lemma}

\begin{proof}
  From the boundary value problems (\ref{eq:bvp_main}) and (\ref{eq:Poisson_A1}), we
  deduce:
  \begin{equation*}
    \left\{
      \arrayoptions{1ex}{1.2}
      \begin{array}{ll}
        -\Delta (u_g-v_f)+g(u_g-v_f)=-gv_f, &\text{in } D, \\
        u_g-v_f=0, &\text{on }\partial D.
      \end{array}
    \right.
  \end{equation*}
  Since $g$ and $v_f$ are non-negative, $u_g<v_f$ in $D$ by the strong
  maximum principle.

  In order to derive a contradiction, we assume that there exists an
  $L\subseteq S(g)$ such that the measure of $L$ is positive, and
  $u_g$ is constant on $L$. By applying Lemma 7.7
  in~\cite{GilbargTrudinger:EllipticPDE_SecOrd_Book:2001}, we
  infer $f=gu_g$ in $L$. Hence:
\begin{equation*}
  f=gu_g<gv_f\le v_f \le f, \text{ in } L,
\end{equation*}
which is a contradiction.
\end{proof}

Next, we turn to the energy functional. In order to prove the
existence and uniqueness of solutions of the minimization problem
(\ref{eq:main_problem}), we need the following basic result regarding
the energy functional $\Phi$:
\begin{lemma}
  \label{lem:basic_props_of_energy_func} The energy functional $\Phi$ satisfies the
  following:
  \begin{enumerate}[(i)]
  \item{\label{item:weakly_continuous}} $\Phi$ is weakly continuous on $\overline{\cal{R}}$ with
    respect to the $L^2-$topology.
  \item{\label{item:stricyly_convex}} $\Phi$ is strictly convex on $\overline{\cal{R}}$.
  \item{\label{item:Phi_deriv}} Given $g$ and $h$ in
    $\overline{\cal{R}}$, the following formula holds:
    \begin{equation}
      \label{eq:Phi_deriv}
      \lim_{t\to
        0^+}\frac{\Phi(\xi_t)-\Phi(g)}{t}=-\int_D(h-g)u^2 \md x,\quad
      0<t<1,
    \end{equation}
    in which $\xi_t=g+t(h-g)$, and $u=u_g$.
  \end{enumerate}
\end{lemma}

\begin{proof}
  The proof of this lemma is quite long and involved. In order not to
  break the flow of the discussion, the proof is placed in a separate
  section altogether. Please see Sect.~\ref{sec:proof-lemma-phi}.
\end{proof}

The main result of the paper is the following:

\begin{theorem}
  \label{thm:main}
  Suppose that $f$ satisfies assumption {\bf A1}. Then the
  minimization problem (\ref{eq:main_problem}) has a unique solution
  $\hat g\in \cal R$. Moreover, there exists an increasing function
  $\psi$ such that:
  \begin{equation}\label{euler1}
    \hat g=\psi(\hat u)\quad \text{a.e. in } D,
  \end{equation}
  where $\hat u=u_{\hat g}$.
\end{theorem}

\begin{proof} 

  We relax the minimization problem (\ref{eq:main_problem}) first by
  extending the admissible set $\cal{R}$ to
  $\overline{\cal{R}}$. Thus, we consider:
  \begin{equation}
    \label{eq:relaxed}
    \inf_{g\in\overline{\cal{R}}}\Phi(g)
  \end{equation}
  By Lemma \ref{lem:basic_props_of_energy_func} (\ref{item:weakly_continuous}), $\Phi$ is
  weakly continuous on $\overline{\cal{R}}$ with respect to the
  $L^2$-topology. Hence, the minimization problem (\ref{eq:relaxed})
  is solvable. Furthermore, thanks to the strict convexity of $\Phi$
  (Lemma \ref{lem:basic_props_of_energy_func} (\ref{item:stricyly_convex})) the solution to
  (\ref{eq:relaxed}) is unique.  Let us denote this solution by
  $\hat{g}$.

  Fix $g\in\overline{\cal{R}}$, and set $g_t=\hat{g}+t(g-\hat{g})$,
  for $t\in (0,1)$. Due to the convexity of $\overline{\cal{R}}$,
  $g_t\in\overline{\cal{R}}$. From Lemma
  \ref{lem:basic_props_of_energy_func}~(\ref{item:Phi_deriv}) we can derive
  $\int_D(g-\hat{g})\hat{u}^2 \md x\le0$. Whence, $\hat{g}$ maximizes the
  linear functional $L(h) \defeq \int_Dh\hat{u}^2 \md x$, relative to
  $h\in\overline{\cal{R}}$. From Lemma \ref{lem:rearrange1} and Lemma
  \ref{lem:estimate}, it follows that the graph of $\hat{u}_S$, the
  restriction of $\hat{u}$ to the set $S(\hat{g})$, has no significant
  flat sections on $S(\hat{g})$. From Lemma \ref{rearrange2}, we know
  that there exists a $g_1\in\cal{R}$ such that $S(g_1)\subseteq
  S(\hat{g})$. Therefore, if we denote by ${\cal R}_S$ the functions
  which are rearrangements of $g_1$ on $S(\hat g)$, then by Lemma
  \ref{lem:zeroflat} we infer the existence of an increasing function
  $\psi_S$ such that $\psi_S({(\hat u)^2}_S)\in{\cal R}_S$. We now
  proceed to extending $\psi_S$ to an increasing function $\psi$ in
  such a way that $\psi (\hat u)\in{\cal R}(g_1)=\cal{R}$. Let us
  assume for the moment that this task has been accomplished. Then,
  from Lemma \ref{lem:uniquemaximizer}, it follows that $\psi(\hat u)$
  is the unique maximizer of the functional $L$, whence we must have
  $\hat g=\psi(\hat u)$, which is the desired result.

  We now come to the issue of extending $\psi_S$. This is done in two
  steps. The first step is to show that $\hat u$ attains its largest
  values on $S(\hat g)$. To this end, it suffices to prove the
  following inequality:
\begin{equation}
  \label{inequality}
  \alpha\equiv {\ess} \inf_{S(\hat g)}\hat u\ge {\ess} \sup_{S(\hat
    g)^c}\hat u\equiv\beta,
\end{equation}
where $S(\hat g)^c$ denotes the complement of $S(\hat g)$. We
prove~(\ref{inequality}) by contradiction. So, let us suppose that
$\alpha<\beta$. Hence, there exist constants $\gamma,\delta,$ and sets
$A\subseteq S(\hat g)$, $B\subseteq S(\hat g)^c$, such that
$\beta>\gamma>\delta>\alpha$, and:
\begin{equation*}
  \left\{
    \arrayoptions{1ex}{1.2}
  \begin{array}{ll}
    \hat u\le\delta & \text{on } A,\\
    \hat u\ge\gamma & \text{on } B .
  \end{array}
  \right.
\end{equation*}
We may assume that $|A|=|B|$, otherwise we consider subsets of $A$ and
$B$. Let $\eta:A\rightarrow B$ be a measure preserving
bijection.\footnote{Such a map exists. See,
  {\eg},~\cite{Royden:RealAnalysis:1988}.} Next, we define a new
function $\overline{g}$ as follows:
\begin{equation*}
  \overline{g}(x)
  =
  \left\{
    \arrayoptions{1ex}{1.2}
    \begin{array}{ll}
      {\hat g}(x) & x\in (A\cup B)^c,\\
      {\hat g}(\eta (x)) & x\in A, \\
      {\hat g}(\eta^{-1}(x)) & x\in B.
    \end{array}
  \right.
\end{equation*}
Clearly $\overline g$ is a rearrangement of $\hat g$. Since $\hat
g\in{\overline{\cal R}}$, it follows from Lemma \ref{lem:rearrange5}
that $\overline{g}\in{\overline{\cal R}}$. Thus:
\begin{multline*}
  \int_D\overline{g}\hat u^2 \md x-\int_D\hat g\hat u^2 \md x = \int_{A\cup B}\overline {g}\hat u^2 \md x-\int_{A\cup B}\hat g\hat u^2 \md x
  =\int_B\overline{g}\hat u^2 \md x-\int_A\hat g\hat u^2 \md x \\
   = \int_B\hat g(\eta^{-1}(x))\hat u^2 \md x-\int_A\hat g\hat u^2 \md x=\int_A\hat g(x)\hat u^2(\eta(x)) \md x-\int_A\hat g\hat u^2 \md x
   \ge (\gamma^2-\delta^2)\int_A\hat g \md x>0,
\end{multline*}
which contradicts the maximality of $\hat g$.

In the second step, we give an explicit formula for the extended
function as follows:
\begin{equation*}
  \hat{\psi} (t) =
  \left\{
    \arrayoptions{1ex}{1.2}
    \begin{array}{ll}
      \psi_S(t) &t>\alpha^2, \\
      0 &t\le\alpha^2,
    \end{array}
  \right.
\end{equation*}
where $\alpha$ is defined in (\ref{inequality}). Clearly, $\hat{\psi}$
is increasing and $\hat{\psi} (\hat u^2)\in{\cal R}(g_1)=\cal{R}$.
Hence, by setting $\psi(t) \defeq \hat{\psi}(t^2)$ we derive
(\ref{euler1}). The proof of the theorem is completed.
\end{proof}

\begin{remark}
  As mentioned earlier, in the special case of $g_0=\chi_{E_0}$ with
  $|E_0|=\alpha$, the minimization problem~(\ref{eq:main_problem})
  reduces to the one considered
  in~\cite{Henrot_Maillot:2001:Optim_Shape_Actuator}. So, $\hat
  g=\chi_{\hat E}$ with $|\hat E|=\alpha$. Hence, from (\ref{euler1})
  we deduce that $\hat E=\{\hat u>\gamma\}$, for some
  $\gamma>0$. Whence, we derive the following boundary value problem:
  \begin{equation}
    \label{eq323}
    \left\{
      \arrayoptions{1ex}{1.2}
      \begin{array}{ll}
        -\Delta\hat u+\hat u\chi_{\{\hat u>\gamma\}}=f(x) &\text{in}\; D, \\
        \hat u=0 &\text{on}\;\partial D.
      \end{array}
    \right.
  \end{equation}
  By setting $U=\hat u-\gamma$, the differential equation in (\ref{eq323}) becomes:
\begin{equation}
\label{eq324}
\Delta U=(U+\gamma-f)\chi_{\{U> 0\}}-f\chi_{\{U \leq 0\}}.
\end{equation}
So, (\ref{eq324}) is an obstacle problem of type:
\begin{equation}
\label{eq325}
\Delta U=G(x)\chi_{\{U > 0\}}-H(x)\chi_{\{U \leq 0\}}, 
\end{equation} 
where $G\leq 0$, because $\hat u\leq f$, and $H(x)\geq 0$. Since $G(x)+H(x)\geq 0$,
we can apply the result of~\cite{Shahgholian:C11_Reg_Elliptic:2003} to deduce that the free boundary has $C\sp {1,1}$ regularity.
\end{remark}

\subsection{Local minimizers}

Even though $\hat g$ in Theorem \ref{thm:main} is a global minimizer,
is it possible for $\Phi$ to have \emph{non-global} local minimizers
over ${\cal R}$? The answer to this question is \emph{negative}. To
prove this, we need a less restrictive version of Theorem 3.3 (iii)
in~\cite{Burton:1989}, stated as follows:

\begin{lemma}
  \label{lem:less_restrictive_Burton}
  Let ${\cal N}: L^r(D)\rightarrow\RR$ be weakly sequentially
  continuous, and let ${\cal R}={\cal R}(h_0)$ denote the rearrangement
  class generated by some $h_0\in L^r(D)$. Assume that for every pair
  $(h_1,h_2)\in \overline{\cal R}\times \overline{\cal R}$ the
  following relation holds:
  \begin{equation*}
    \lim_{t\rightarrow 0^+}\frac{{\cal N}(th_2+(1-t)h_1)-{\cal
        N}(h_1)}{t} = \int_D(h_2-h_1) \, {\cal G} \,  \md x,
  \end{equation*}
  for some ${\cal G}\in L^{r^\prime}(D)$. Suppose ${\cal U}$ is a
  strong neighborhood (relative to ${\cal R}$) of $\hat h\in {\cal
    R}$, for which we have:
  \begin{equation*}
    \forall h\in{\cal U}: {\cal N}(\hat h)\leq\cal{N}(h).
  \end{equation*}
  Then, $\hat h$ minimizes the linear functional ${\cal
    L}(h) \defeq \int_D h \, {\cal G} \,  \md x$, relative to $h\in\overline{\cal R}$.
\end{lemma}

Now we state our result concerning local minimizers.
\begin{theorem}
  \label{thm:local_minim}
  Let the hypotheses of Theorem \ref{thm:main} hold. If $g_1$ and
  $g_2$ are two local minimizers of $\Phi(g)$ relative to $g\in {R}$,
  then $g_1=g_2$.
\end{theorem}

\begin{proof}
  For simplicity we set $u_1 \defeq u_{g_1}$ and $u_2 \defeq
  u_{g_2}$. Lemma \ref{lem:less_restrictive_Burton}, in conjunction
  with Lemma \ref{lem:basic_props_of_energy_func}~(\ref{item:Phi_deriv}), implies that $g_1$
  and $g_2$ are maximizers of the linear functionals:
  \begin{equation*}
  {\cal L}_1(g) \defeq \int_Dgu_1^2 \md x,
  \end{equation*}
  and
  \begin{equation*}
  {\cal L}_2(g) \defeq \int_Dgu_2^2 \md x,
  \end{equation*}
  relative to $g\in\overline{\cal R}$, respectively. In particular, we
  infer:
  \begin{equation}
    \label{eq3456}
    \int_Dg_2u_1^2 \md x\leq\int_Dg_1u_1^2 \md x \quad \text{and}\quad \int_Dg_1u_2^2 \md x\leq\int_Dg_2u_2^2 \md x.
  \end{equation}
  Thus, we obtain:
  \arrayoptions{0.5ex}{1.1}
  \begin{eqnarray}
    \label{eq34567}
    2\int_Dfu_1 \md x-\int_D(|\nabla u_1|^2+g_1u_1^2) \md x&\leq&2\int_Dfu_1 \md x-\int_D(|\nabla u_1|^2+g_2u_1^2) \md x \nonumber \\
    &\leq&2\int_Dfu_2 \md x-\int_D(|\nabla u_2|^2+g_2u_2^2) \md x \nonumber \\
    &\leq& 2\int_Dfu_2 \md x-\int_D(|\nabla u_2|^2+g_1u_2^2) \md x \nonumber \\
    &\leq& 2\int_Dfu_1 \md x-\int_D(|\nabla u_1|^2+g_1u_1^2) \md x,
  \end{eqnarray}
  where the first and third inequalities are consequences of
  (\ref{eq3456}), whereas the second and the fourth inequalities
  follow from (\ref{eq:variational_formulation}). From (\ref{eq34567}) we see that all
  inequalities must in fact be equalities. This, in turn, implies that
  $u_1=u_2$, due to the uniqueness. Whence, we deduce $g_1=g_2$ as
  desired.
\end{proof}

\subsection{Radial domain}

Here we present our result regarding radial symmetry of the
optimizers. Note how, compared with similar results in the literature,
in our approach, such result may be obtained with minimal
technicalities:

\begin{theorem}
  \label{them:radial_sym}
  Suppose that $f$ is radial and satisfies assumption {\bf
    A1}. Then the solution of (\ref{eq:main_problem}) is radial and non-increasing.
\end{theorem}

\begin{proof}
  Let $g$ denote the solution of (\ref{eq:main_problem}) and let $R$
  be a rotational map about the origin. Since $f$ is radial, we infer
  $u_g\circ R=u_{g\circ R}$. Thus, $\Phi (g\circ R)=\Phi (g)$, and
  $g\circ R$ is also a solution of (\ref{eq:main_problem}). By
  uniqueness, we deduce $g\circ R=g$, for every rotational map
  $R$. Whence, $g$ is radial, as desired. To prove that $g$ is
  non-increasing, we observe that, since $u=u_g$ is radial, we can
  write the equation in (\ref{eq:bvp_main}) as:
$$-(r^{N-1}u')'=r^{N-1}(f-gu).$$
Since $f\ge v_f$ by {\bf A1}, and $g\le 1$ by assumption, we have
$f-gu\ge v_f-u$. Furthermore, $v_f-u>0$ by the proof of Lemma
\ref{lem:estimate}. Hence,
$$-(r^{N-1}u')'>0,\ \ \ -r^{N-1}u'>0,\ \ \ \ u'<0.$$
By Theorem~\ref{thm:main}, $g=\psi(u)$, for some non-decreasing
$\psi$. As a result, $g$ is non-increasing, as desired.
\end{proof}

\subsection{Some remarks on maximization}

In addition to the minimization problem (\ref{eq:main_problem}), one
can also consider the maximization problem:
\begin{equation}
\label{maximization}
\sup_{g\in{\cal R}}\Phi (g).
\end{equation}
Since $\Phi$ is weakly continuous and convex, $\Phi$ reaches its
maximum value at the extremal points of the convex set
$\overline{\mathcal R}$ ({\ie}, the elements of $\mathcal R$). Hence,
problem (\ref{maximization}) is solvable (see Theorem 7 of
\cite{Burton:1987} or Remark 3.1 of
\cite{Henrot_Maillot:2001:Optim_Shape_Actuator}). Moreover, if
the assumption {\bf A1} holds, along the same lines as in the proof of
Theorem~\ref{thm:main}, it can be shown that, if $\tilde g$ is a
maximizer, then:
\begin{equation}\label{eq321}
\tilde g=\tilde\psi (\tilde u),
\end{equation}
almost everywhere in $D$, for some decreasing function $\tilde\psi$.
Here $\tilde u=u_{\tilde g}$, the solution of (\ref{eq:bvp_main}) with
$g=\tilde g$.

Note that, for maximizers we do not have uniqueness in general.
However, we are going to prove that, in case $D$ is a ball and $f$ is
radially symmetric and non-increasing, any maximizer is radially
symmetric and non-decreasing, hence unique. Indeed, let
$v=u_{\tilde{g}_*}$, where $\tilde{g}_*$ is the increasing Schwarz
symmetrization of $\tilde g$ (see~\cite{Kawohl:1985}). For
simplicity, we write $u$ instead of $u_{\tilde g}$. By
Lemma~\ref{lem:Dirichlet}~(\ref{item:dirich_integ_nabla_vstar_lt_v}):
\begin{equation}
  \label{eq500}
  -\frac{1}{2}\Phi (\tilde g)=\frac{1}{2}\int_D|\nabla
  u|^2 \md x+\frac{1}{2}\int_D\tilde g u^2 \md x-\int_Dfu \md x
  \geq \frac{1}{2}\int_D|\nabla u^*|^2 \md x+\frac{1}{2}\int_D\tilde g
  u^2 \md x-\int_Dfu \md x.
\end{equation}

Now, by applying the Hardy-Littlewood inequality to the last two
integrals in (\ref{eq500}), keeping in mind that $f=f^*$, we obtain:
\begin{equation}
  \label{eq501}
  -\frac{1}{2}\Phi (\tilde g)\geq \frac{1}{2}\int_D|\nabla
  u^*|^2 \md x+\frac{1}{2}\int_D{\tilde g}_*\, {u^*}^2 \md x-\int_Dfu^* \md x.
\end{equation}
Recalling that $v$ minimizes the functional
$${\cal I}(w)=\frac{1}{2}\int_D|\nabla w|^2 \md x+\frac{1}{2}\int_D{\tilde
  g}_* w^2 \md x-\int_Dfw \md x,$$
relative to $w\in H^1_0(D)$, we infer from (\ref{eq501}) that:
\begin{equation}
\label{eq502}
-\frac{1}{2}\Phi (\tilde g)\geq \frac{1}{2}\int_D|\nabla v|^2 \md x+\frac{1}{2}\int_D{\tilde g}_* v^2 \md x-\int_Dfv \md x=-\frac{1}{2}\Phi ({\tilde g}_*).
\end{equation}
As $\tilde g$ is maximal for $\Phi$, then $\Phi ({\tilde g}_*)\leq\Phi
(\tilde g)$, which together with (\ref{eq500}), (\ref{eq501}) and
(\ref{eq502}) yield:
\begin{equation*}
\int_D|\nabla u|^2 \md x=\int_D|\nabla u^*|^2 \md x.
\end{equation*}

Thus, from
Lemma~\ref{lem:Dirichlet}~(\ref{item:dirich_v_trans_vStar}), we see
that $u^{-1}(\alpha, \infty)$ is a ball for every $0\leq\alpha<M=
{\ess} \sup u.$ We now proceed to show that $u=u^*$. Recalling
Lemma~\ref{lem:Dirichlet}~(\ref{item:dirich_v_trans_vStar}), it
suffices to verify that the set $\{x\in D:\;\nabla u=0,\;0<u(x)<M\}$
is measure zero. To this end, consider $x_0\in D$, and set $S \defeq
\{u\geq u(x_0)\}$. We know that $S$ is a disk (ball), and by
continuity of $u$, $x_0\in\partial S\subseteq\{u=u(x_0)\}$.  So we can
apply the Hopf lemma (see,
{\eg},~\cite{Fraenkel:Maximum_Principles:Book:2000}), and deduce
that $\frac{\partial u}{\partial\nu}(x_0)<0$, where $\nu$ denotes the
unit outward normal vector to $\partial S$ at $x_0$. Whence, in
particular, $\nabla u(x_0)\neq 0$. Thus, in fact, $\{x\in D:\;\nabla
u=0,\;0<u(x)<M\}$ is empty, so its measure is zero, as desired. This
implies $u=u^*$, and by (\ref{eq321}), $\tilde g=\tilde \psi (u^*)$
almost everywhere in $D$. Since $\tilde \psi$ is decreasing, $\tilde
g$ is radial and non-decreasing, as claimed.

\begin{remark}
  A consequence of (\ref{euler1}) is that the larger values of $\hat
  u$ are attained where $\hat g$ is large. Whence, in case the set
  $\{\hat g=0\}$ has positive measure, it will contain a layer around
  the boundary $\partial D$, since $\hat u$ is continuous, and
  vanishes on $\partial D$. \emph{Physically, this means that in the
    construction of a robust membrane one should use the material with
    least density near the boundary}. The dual conclusion can be drawn
  similarly regarding the maximization problem (\ref{maximization}).
\end{remark}

\begin{remark}
  Note that Theorem~\ref{them:radial_sym} can be improved. Indeed, if
  $D$ is Steiner symmetric with respect to a hyperplane $l$ (see, {\eg},~\cite{Kawohl:1985}), then $\hat g$ (the solution of
  (\ref{eq:main_problem})) will also be Steiner symmetric with respect to $l$. Of
  course, in this case, one needs to use the inequality:
  \begin{equation*}
    \int_D|\nabla u|^2 \md x\geq\int_D|\nabla u^\sharp|^2 \md x,
  \end{equation*}
  instead of (\ref{eq:PZ}), in which $u^\sharp$ stands for the Steiner
  symmetrization of $u$. A similar result can be obtained for the
  maximization problem (\ref{maximization}). Of course, for the
  maximization problem we do not necessarily have uniqueness of
  optimal solutions.
\end{remark}

\section{Shape optimization}
\label{sec:shape-optimization}

In this section, we focus on the shape optimization variant of our
main problem, {\ie}, the case where the generator $g_0$ is
two-valued. Thus, we consider the following boundary value problem:

\begin{equation}
  \label{eq:bvp_alpha_beta}
  \left\{
    \arrayoptions{1ex}{1.2}
    \begin{array}{ll}
      -\Delta u+(\alpha\chi_E +\beta\chi_{E^c}) \, u=f, &\text{in}\;D, \\
      u=0, &\text{on}\;\partial D,
    \end{array}
  \right.
\end{equation}
in which, $D$ is a smooth bounded domain in $\RR^N$, $N \in \{ 2,3
\}$, $f \in L^2(D)$ is a given non-negative function,
$1\ge\alpha>\beta\ge 0$, $E$ is a measurable subset of $D$, and $E^c$
is the complement of $E$ in $D$.\footnote{To see why the assumption $1
  \geq \alpha$ is imposed, see Lemma~\ref{lem:estimate}
  \vpageref[above]{lem:estimate}.}  Denoting the unique solution of
(\ref{eq:bvp_alpha_beta}) by $u_{E}$, we are interested in the
following minimization problem:
\begin{equation}
  \label{eq:min_prob_alpha_beta}
  \inf_{|E|=\gamma}\int_Dfu_E \md x,
\end{equation}
where $0<\gamma<|D|$. By Theorem \ref{thm:main}, we know that, if $f$
satisfies {\bf A1}, then (\ref{eq:min_prob_alpha_beta}) has a unique
solution $\tilde{D}\subset D$, with $|\tilde{D}|=\gamma$. Also, we
have $\tilde{D}= \left\{x\in D:u_{\tilde{D}}(x)>c\right\}$, for some
positive $c$, which, in turn, implies:
\begin{equation*}
u_{\tilde{D}}(x)=c,\quad \text{on}\;\partial\tilde{D}.
\end{equation*}

Our aim is to analyze monotonicity and stability of solutions with
respect to the parameters $\alpha$ and $\gamma$. Analyses of this kind
are crucial for laying the foundation for computable analysis of shape
optimization problems such as (\ref{eq:bvp_alpha_beta}). 

\begin{remark}
In what follows, we keep the presentation succinct, and as such, many
of the claims will be listed with the proofs omitted. The interested
reader may refer to Sect.~3.3 of~\cite{Liu:PhD_Thesis:2015} for
the details of the omitted proofs. Nonetheless, we present the proofs
of a few of the more interesting cases.
\end{remark}

 \subsection{Monotonicity and stability results with respect to
   $\boldsymbol{\gamma}$}

 We know that, for each $0 < \gamma < |D|$, the minimization problem
 (\ref{eq:min_prob_alpha_beta}) has a unique solution. Now, consider
 $0<\gamma_1,\gamma_2<|D|$, and their corresponding unique solutions:
 \begin{equation}
   \label{eqa1}
   \tilde{D}_{\gamma_1}=\left\{x\in D: u_{\gamma_1}(x)>c_{\gamma_1}\right\}\quad\text{and}\quad
   \tilde{D}_{\gamma_2}=\left\{x\in D:
     u_{\gamma_2}(x)>c_{\gamma_2}\right\},
 \end{equation}
 for some positive $c_{\gamma_1}$ and $c_{\gamma_2}$, where
 $u_{\gamma_1}$ and $u_{\gamma_2}$ satisfy:
 \begin{equation}
   \label{bvpa3}
   \left\{
     \arrayoptions{1ex}{1.2}
     \begin{array}{ll}
       -\Delta u_{\gamma_1}+(\alpha\chi_{\tilde{D}_{\gamma_1}}
       +\beta\chi_{\tilde{D}_{\gamma_1}^c})u_{\gamma_1}=f, &\text{in}\;D, \\
       u_{\gamma_1}=0, &\text{on}\;\partial D,
     \end{array}
   \right.
\end{equation}
and
\begin{equation}
  \label{bvpa4}
  \left\{
    \arrayoptions{1ex}{1.2}
    \begin{array}{ll}
      -\Delta u_{\gamma_2}+(\alpha\chi_{\tilde{D}_{\gamma_2}}
      +\beta\chi_{\tilde{D}_{\gamma_2}^c})u_{\gamma_2}=f, &\text{in}\;D, \\
      u_{\gamma_2}=0, &\text{on}\;\partial D.
    \end{array}
  \right.
\end{equation}
We also restate the minimization problem
(\ref{eq:min_prob_alpha_beta}), with $\gamma$ as an input parameter:
\begin{equation}
  \label{mpa2}
  \Psi(\gamma) \defeq \inf_{|E|=\gamma}\int_Dfu_E \md x.
\end{equation}

\begin{proposition}
  \label{prop:monot_gamma}
  If $0<\gamma_1<\gamma_2<|D|$, then
  \begin{enumerate}[(i)]
  \item \label{item:MRa1} $c_{\gamma_1}\ge c_{\gamma_2}$.

  \item \label{item:MRa2} $\tilde{D}_{\gamma_1}\subseteq
    \tilde{D}_{\gamma_2}$.
    
  \item \label{item:MRa3} $u_{\gamma_1}> u_{\gamma_2}$ in $D$.
  
  \end{enumerate}
  
\end{proposition}

Since $f$ is non-negative and non-trivial, the following is an easy
consequence of Proposition \ref{prop:monot_gamma} (\ref{item:MRa3}).

\begin{corollary}
  \label{cor:MRa4}
  $\Psi(\gamma)$ is a decreasing function on $(0,|D|)$.
\end{corollary}

\begin{theorem}
  \label{thm:continuity_gamma}
  If $\gamma_1$ tends to $\gamma_2$ in $(0,|D|)$, then $u_{\gamma_1}$
  converges to $u_{\gamma_2}$ in $C(\bar{D})$. Moreover,
  $c_{\gamma_1}$ converges to $c_{\gamma_2}$, where
  $c_{\gamma_1}=u_{\gamma_1}(\partial \tilde{D}_{\gamma_1})$ and
  $c_{\gamma_2}=u_{\gamma_2}(\partial \tilde{D}_{\gamma_2})$.
\end{theorem}

\begin{corollary}
  \label{cor:continuity_Psi_gamma}
  $\Psi(\gamma)$ is continuous on $(0,|D|)$.
\end{corollary}

  From Corollary \ref{cor:MRa4} we infer that $\Psi(\gamma)$ is
  differentiable almost everywhere. However, the following theorem
  shows that it is actually continuously differentiable on $(0,|D|)$.

\begin{theorem}
  \label{thm:differentiability_Psi_gamma}
  $\Psi(\gamma)$ is continuously differentiable on $(0,|D|)$.
  Moreover:
  \begin{equation*}
\Psi^{'}(\gamma)=-(\alpha-\beta)c_\gamma^2  ,
  \end{equation*}
  in which $c_\gamma=u_{\gamma}(\partial \tilde{D}_\gamma)$.
\end{theorem}

\begin{proof} Fix $0<\gamma_2<|D|$, and let $\gamma_1$ increase to
  $\gamma_2$. We claim that
  $\frac{\Psi(\gamma_1)-\Psi(\gamma_2)}{\gamma_1-\gamma_2}$ converges
  to $-(\alpha-\beta)c_{\gamma_2}^2$. From (\ref{bvpa3}) and (\ref{bvpa4}), we deduce
  \begin{equation}
    \label{bvpa5}
    \left\{
      \arrayoptions{1ex}{1.2}
      \begin{array}{ll}
        -\Delta (u_{\gamma_1}-u_{\gamma_2})+(\alpha\chi_{\tilde{D}_{\gamma_1}}
        +\beta\chi_{\tilde{D}_{\gamma_1}^c})(u_{\gamma_1}-u_{\gamma_2})=
        -(\alpha-\beta)u_{\gamma_2}(\chi_{\tilde{D}_{\gamma_1}}-\chi_{\tilde{D}_{\gamma_2}}) &\text{in}\;D \\
        u_{\gamma_1}-u_{\gamma_2}=0 &\text{on}\;\partial D.
      \end{array}
    \right.
  \end{equation}
  Multiplying the differential equation in (\ref{bvpa5}) by
  $u_{\gamma_1}+u_{\gamma_2}$, integrating the result over $D$,
  followed by an application of divergence theorem, in conjunction
  with $\tilde{D}_{\gamma_1}\subseteq \tilde{D}_{\gamma_2}$
  (Proposition \ref{prop:monot_gamma} (\ref{item:MRa2})) yields:
{\arrayoptions{0.5ex}{1.1}
\begin{eqnarray}
  \label{eqa11}
  \int_D(|\nabla u_{\gamma_1}|^2-|\nabla u_{\gamma_2}|^2 )\,  \md x&+&\int_D(\alpha\chi_{\tilde{D}_{\gamma_1}} +
  \beta\chi_{\tilde{D}_{\gamma_1}^c})(u_{\gamma_1}^2-u_{\gamma_2}^2)\,  \md x \nonumber \\ &=&-(\alpha-\beta)\int_D u_{\gamma_2}(u_{\gamma_1}+
  u_{\gamma_2})(\chi_{\tilde{D}_{\gamma_1}}-\chi_{\tilde{D}_{\gamma_2}})\,  \md x\nonumber \\
  &=& (\alpha-\beta)\int_D u_{\gamma_2}(u_{\gamma_1}+u_{\gamma_2})\chi_{\tilde{D}_{\gamma_2}\setminus\tilde{D}_{\gamma_1}}\,  \md x\nonumber \\&=
  &
  (\alpha-\beta)\int_{\tilde{D}_{\gamma_2}\setminus\tilde{D}_{\gamma_1}}
  u_{\gamma_2}(u_{\gamma_1}+u_{\gamma_2})\,  \md x.
\end{eqnarray}}
Furthermore, from (\ref{bvpa3}), (\ref{bvpa4}), and (\ref{eqa11}), we
deduce:
\begin{multline}
  \label{eqa12}
  \Psi(\gamma_1)-\Psi(\gamma_2)=\int_D u_{\gamma_1}f\,  \md x-\int_D u_{\gamma_2}f\,  \md x \\
   = \left[\int_D|\nabla u_{\gamma_1}|^2 \,  \md x +\int_D(\alpha
    \chi_{\tilde{D}_{\gamma_1}}+ \beta
    \chi_{\tilde{D}_{\gamma_1}^c})u_{\gamma_1}^2\,  \md x\right] -\left[\int_D|\nabla
    u_{\gamma_2}|^2 \,  \md x  +\int_D(\alpha \chi_{\tilde{D}_{\gamma_2}}+
    \beta
    \chi_{\tilde{D}_{\gamma_2}^c})u_{\gamma_2}^2\,  \md x\right] \\
  = \int_D(|\nabla u_{\gamma_1}|^2-|\nabla u_{\gamma_2}|^2 )\,  \md x
  +
  \int_D(\alpha\chi_{\tilde{D}_{\gamma_1}}+\beta\chi_{\tilde{D}_{\gamma_2}^c})(u_{\gamma_1}^2-u_{\gamma_2}^2)\,
   \md x + \int_{\tilde{D}_{\gamma_2}\setminus \tilde{D}_{\gamma_1}}(\beta
  u_{\gamma_1}^2-\alpha u_{\gamma_2}^2)\,  \md x \\
 = (\alpha-
  \beta)\int_{\tilde{D}_{\gamma_2}\setminus\tilde{D}_{\gamma_1}}
  u_{\gamma_2}(u_{\gamma_1}+u_{\gamma_2})\,  \md x-\beta
  \int_{\tilde{D}_{\gamma_2}\setminus\tilde{D}_{\gamma_1}}(u_{\gamma_1}^2-u_{\gamma_2}^2)\,
   \md x +\int_{\tilde{D}_{\gamma_2}\setminus
    \tilde{D}_{\gamma_1}}(\beta u_{\gamma_1}^2-\alpha
  u_{\gamma_2}^2)\,  \md x \\
  = (\alpha-\beta)\int_{\tilde{D}_{\gamma_2}\setminus
    \tilde{D}_{\gamma_1}} u_{\gamma_2}u_{\gamma_1}\,  \md x.
\end{multline}
where we have used the fact that $\tilde{D}_{\gamma_1}\subseteq
\tilde{D}_{\gamma_2}$ in the third and fourth equality, and also
applied (\ref{eqa11}) in the fourth equality. By using (\ref{eqa12})
and the fact that
$|\tilde{D}_{\gamma_2}\setminus\tilde{D}_{\gamma_1}|=\gamma_2-\gamma_1$,
we calculate:
\arrayoptions{0.5ex}{1.3}
\begin{eqnarray}
  \label{eqa13}
  \left|\frac{\Psi(\gamma_1)-\Psi(\gamma_2)}{\gamma_1-\gamma_2}-\left[-(\alpha-\beta)c_{\gamma_2}^2\right]\right|
  &=&\frac{\alpha-\beta}{\gamma_2-\gamma_1}\left|\int_{\tilde{D}_{\gamma_2}\setminus\tilde{D}_{\gamma_1}} (u_{\gamma_2}u_{\gamma_1}-
    c_{\gamma_2}^2) \md x\right|\nonumber \\ &\le&
  (\alpha-\beta)\left\|u_{\gamma_2}u_{\gamma_1}-c_{\gamma_2}^2\right\|
  _{\infty,\tilde{D}_{\gamma_2}\setminus\tilde{D}_{\gamma_1}}.
\end{eqnarray}
By (\ref{eqa1}) and Proposition \ref{prop:monot_gamma} (\ref{item:MRa3}),
in $\tilde{D}_{\gamma_2}\setminus\tilde{D}_{\gamma_1}$ we have
$c_{\gamma_2}< u_{\gamma_2}<u_{\gamma_1}\le c_{\gamma_1}$. So, by
applying Theorem \ref{thm:continuity_gamma} we infer:
\begin{equation*}
  \left\|u_{\gamma_2}u_{\gamma_1}-c_{\gamma_2}^2\right\|
  _{\infty,\tilde{D}_{\gamma_2}\setminus\tilde{D}_{\gamma_1}}\le |c_{\gamma_1}^2-c_{\gamma_2}^2|,
\end{equation*}
which converges to zero. From (\ref{eqa13}), we obtain the desired
result.

Similarly, when $\gamma_1$ decreases to $\gamma_2$, the ratio
$\frac{\Psi(\gamma_1)- \Psi(\gamma_2)}{\gamma_1-\gamma_2}$ converges
to $-(\alpha-\beta)c_{\gamma_2}^2$. By Theorem \ref{thm:continuity_gamma}
we know that $c_\gamma$ is continuous with respect to $\gamma$. Hence,
we infer that $\Psi(\gamma)$ is continuously differentiable with
$\Psi^{'}(\gamma)=-(\alpha-\beta)c_\gamma^2$ on $(0,|D|)$.
\end{proof}

\subsection{Monotonicity and stability results with respect to $\alpha$}

Assume that $0\le\beta<\alpha_1,\alpha_2\le 1$. For each of $\alpha_1$
and $\alpha_2$, the minimization problem (\ref{eq:min_prob_alpha_beta}) has a unique
solution, which we denote by $\tilde{D}_{\alpha_1}$ and
$\tilde{D}_{\alpha_2}$, respectively. We know that
$|\tilde{D}_{\alpha_1}|=|\tilde{D}_{\alpha_2}|=\gamma$, and:
\begin{equation}
  \label{eqa14}
  \tilde{D}_{\alpha_1}=\left\{x\in D: u_{\alpha_1}(x)>c_{\alpha_1}\right\}\quad\text{and}\quad \tilde{D}_{\alpha_2}=
  \left\{x\in D: u_{\alpha_2}(x)>c_{\alpha_2}\right\},
\end{equation}
for some positive $c_{\alpha_1}$ and $c_{\alpha_2}$, where
$u_{\alpha_1}$ and $u_{\alpha_2}$ satisfy:
\begin{equation}
  \label{bvpa6}
  \left\{
    \arrayoptions{0.5ex}{1.2}
    \begin{array}{ll}
      -\Delta
      u_{\alpha_1}+(\alpha_1\chi_{\tilde{D}_{\alpha_1}} +
      \beta\chi_{\tilde{D}_{\alpha_1}^c})u_{\alpha_1}=f, &\text{in}\;D, \\
      u_{\alpha_1}=0, &\text{on}\;\partial D,
    \end{array}
  \right.
\end{equation}
and
\begin{equation}
  \label{bvpa7}
  \left\{
    \arrayoptions{0.5ex}{1.2}
    \begin{array}{ll}
     -\Delta
      u_{\alpha_2}+(\alpha_2\chi_{\tilde{D}_{\alpha_2}} +
      \beta\chi_{\tilde{D}_{\alpha_2}^c})u_{\alpha_2}=f, &\text{in}\;D, \\
      u_{\alpha_2}=0, &\text{on}\;\partial D.
    \end{array}
  \right.
\end{equation}
%


This time, we restate the minimization problem
(\ref{eq:min_prob_alpha_beta}), with $\alpha$ as an input parameter:
\begin{equation}
  \label{mpa3}
  \Psi(\alpha) \defeq \inf_{|E|=\gamma}\int_Dfu_{E,\alpha}\,  \md x=\int_Dfu_{\alpha}
  \,  \md x.
\end{equation}

\begin{proposition}
\label{prop:monot-alpha}
  If $0<\alpha_1<\alpha_2\le 1$, then:

  \begin{enumerate}[(i)]
  \item \label{item:MRa5} $c_{\alpha_1}> c_{\alpha_2}$.
  \item $u_{\alpha_1}>u_{\alpha_2}$ in $D$.
  \item $\tilde{D}_{\alpha_1}\cap \tilde{D}_{\alpha_2}\neq \emptyset$.
  \end{enumerate}
\end{proposition}

\begin{theorem}
  \label{thm:stabilitya1}
  If $\beta>0$ and $\alpha_1$ converges to $\alpha_2$ in $(\beta,1]$,
  then $|\tilde{D}_{\alpha_1}\vartriangle\tilde{D}_{\alpha_2}|$
  converges to zero.
\end{theorem}

\begin{proof} Fix $\beta<\alpha_2\le1$ and let $\alpha_1$ increase to
  $\alpha_2$.  We claim that
  $|\tilde{D}_{\alpha_1}\vartriangle\tilde{D}_{\alpha_2}|$ converges
  to zero. First, let us introduce the following auxiliary boundary
  value problem
  \begin{equation}
    \label{bvpa10}
    \left\{
      \arrayoptions{0.5ex}{1.3}
      \begin{array}{ll}
        -\Delta
        \hat{u}_{\alpha_1}+(\alpha_1\chi_{\tilde{D}_{\alpha_2}} +
        \beta\chi_{\tilde{D}_{\alpha_2}^c})\hat{u}_{\alpha_1}=f &\text{in}\;D, \\
        \hat{u}_{\alpha_1}=0 &\text{on}\;\partial D.
      \end{array}
    \right.
\end{equation}
From (\ref{bvpa7}) and (\ref{bvpa10}), we deduce:
\begin{equation}
  \label{bvpa11}
  \left\{
    \arrayoptions{0.5ex}{1.3}
    \begin{array}{ll} -\Delta
      (\hat{u}_{\alpha_1}-u_{\alpha_2})+
      (\alpha_1\chi_{\tilde{D}_{\alpha_2}} +\beta\chi_{\tilde{D}_{\alpha_2}^c})(\hat{u}_{\alpha_1}-u_{\alpha_2})=
      (\alpha_2-\alpha_1)u_{\alpha_2}\chi_{\tilde{D}_{\alpha_2}} &\text{in}\;D, \\
      \hat{u}_{\alpha_1}-u_{\alpha_2}=0 &\text{on}\;\partial D.
    \end{array}
  \right.
\end{equation}
Since $\alpha_2>\alpha_1$, we infer that
$(\alpha_2-\alpha_1)u_{\alpha_2}\chi_{\tilde{D}_{\alpha_2}}$ is
non-negative.  So, by applying the strong maximum principle to
(\ref{bvpa11}), we obtain $\hat{u}_{\alpha_1}>u_{\alpha_2}$ in
$D$. Furthermore, by (\ref{eqa14}), we have:
\begin{equation}
\label{eqa20}
\hat{D}_{\alpha_1} = \left\{x\in D: \hat{u}_{\alpha_1}(x)>c_{\alpha_2}\right\}\supseteq \left\{x\in D:
  u_{\alpha_2}(x)>c_{\alpha_2}\right\}=\tilde{D}_{\alpha_2}.
\end{equation}
Multiplying the differential equation in (\ref{bvpa11}) by
$\hat{u}_{\alpha_1}-u_{\alpha_2}$, integrating the result over $D$,
followed by an application of divergence theorem yields:
\arrayoptions{0.5ex}{1.3}
\begin{eqnarray}
  \label{eqa21}
  \int_D|\nabla(\hat{u}_{\alpha_1}-u_{\alpha_2})|^2 dx&+&\int_D(\alpha_1\chi_{\tilde{D}_{\alpha_2}} +
  \beta\chi_{\tilde{D}_{\alpha_2}^c})(\hat{u}_{\alpha_1}-u_{\alpha_2})^2dx\nonumber \\ &=&(\alpha_2-\alpha_1)\int_D u_{\alpha_2}(\hat{u}_{\alpha_1}-
  u_{\alpha_2})\chi_{\tilde{D}_{\alpha_2}}dx\nonumber \\
  &\le & (\alpha_2-\alpha_1)\left\|u_{\alpha_2}\right\|_4\left\|\hat{u}_{\alpha_1}
    -u_{\alpha_2}\right\|_4|\tilde{D}_{\alpha_2}|^{\frac{1}{2}}\nonumber\\
  &\le & C(\alpha_2-\alpha_1)\left\|u_{\alpha_2}\right\|_{H_0^1(D)}\left\|\hat{u}_{\alpha_1}
    -u_{\alpha_2}\right\|_{H_0^1(D)}|\tilde{D}_{\alpha_2}|^{\frac{1}{2}},
\end{eqnarray}
where we have used general H{\"o}lder's inequality in the first
inequality, and Sobolev embedding theorem in the second
inequality. Since the second term of the first line of (\ref{eqa21})
is non-negative, we obtain:
\begin{equation}
  \label{eqa22}
  \left\|\hat{u}_{\alpha_1}
    -u_{\alpha_2}\right\|_{H_0^1(D)}\le
  C(\alpha_2-\alpha_1)\left\|u_{\alpha_2}\right\|_{H_0^1(D)}
  |\tilde{D}_{\alpha_2}|^{\frac{1}{2}}.
\end{equation}
Noting that $\alpha_1$ increases to $\alpha_2$, we infer
$\hat{u}_{\alpha_1}$ converges to $u_{\alpha_2}$ in $H_0^1(D)$.  By
using elliptic regularity theory and Sobolev embedding theorem, we
infer $\hat{u}_{\alpha_1}$ converges to $u_{\alpha_2}$ in
$C(\bar{D})$. So, from (\ref{eqa20}) and the fact that
$|\tilde{D}_{\alpha_2}|=\gamma$, in conjunction with Lemma
\ref{lem:estimate}, we deduce that
$|\hat{D}_{\alpha_1}\setminus\tilde{D}_{\alpha_2}|$ decreases to zero,
and
\begin{equation}
  \label{eqa24}
  |\hat{D}_{\alpha_1}| \to \gamma^+.
\end{equation}
On the other hand, from (\ref{bvpa6}) and (\ref{bvpa10}), we have:
\begin{equation}
  \label{bvpa12}
  -\Delta
  (u_{\alpha_1}-\hat{u}_{\alpha_1})+(\alpha_1\chi_{\tilde{D}_{\alpha_1}}
  +
  \beta\chi_{\tilde{D}_{\alpha_1}^c})(u_{\alpha_1}-\hat{u}_{\alpha_1})
  =(\alpha_1-\beta)\hat{u}_{\alpha_1}
  (\chi_{\tilde{D}_{\alpha_2}\setminus\tilde{D}_{\alpha_1}}-\chi_{\tilde{D}_{\alpha_1}\setminus\tilde{D}_{\alpha_2}})\quad
  \text{in } D,
\end{equation}
with $u_{\alpha_1}-\hat{u}_{\alpha_1}=0$ on $\partial D$. Now, let us
introduce the following subsets of $D$:
\begin{equation*}
  \left\{
  \arrayoptions{0.5ex}{1.3}
  \begin{array}{lcl}
    \hat{E} & \defeq &\left\{x\in D:u_{\alpha_1}(x)-\hat{u}_{\alpha_1}(x)\le c_{\alpha_1}-c_{\alpha_2}\right\},\\
\hat{F} & \defeq &\left\{x\in D:u_{\alpha_1}(x)-\hat{u}_{\alpha_1}(x)> c_{\alpha_1}-c_{\alpha_2}\right\}.
  \end{array}
  \right.
\end{equation*}
Using (\ref{eqa14}) and (\ref{eqa20}), we infer
$\tilde{D}_{\alpha_1}\setminus\hat{D}_{\alpha_1}\subseteq \hat{F}$ and
$\hat{D}_{\alpha_1}\setminus\tilde{D}_{\alpha_1}\subseteq
\hat{E}$. Moreover, by (\ref{eqa20}), we have $\hat{F}=
(\hat{E})^c\subseteq(\hat{D}_{\alpha_1}\setminus\tilde{D}_{\alpha_1})^c\subseteq
(\tilde{D}_{\alpha_2}\setminus\tilde{D}_{\alpha_1})^c$. So,
(\ref{bvpa12}) leads to:
\begin{equation}
  \label{eqa23}
  -\Delta
  (u_{\alpha_1}-\hat{u}_{\alpha_1})+(\alpha_1\chi_{\tilde{D}_{\alpha_1}}
  +
  \beta\chi_{\tilde{D}_{\alpha_1}^c\cap \tilde{D}_{\alpha_2}^c})(u_{\alpha_1}-\hat{u}_{\alpha_1})
  =-(\alpha_1-\beta)\hat{u}_{\alpha_1}
  \chi_{\tilde{D}_{\alpha_1}\setminus\tilde{D}_{\alpha_2}}
  \quad\text{in }
  \hat{F}\subseteq(\tilde{D}_{\alpha_2}\setminus\tilde{D}_{\alpha_1})^c.
\end{equation}
Since $c_{\alpha_1}> c_{\alpha_2}$ (by
Proposition~\ref{prop:monot-alpha} (\ref{item:MRa5})), we have
$u_{\alpha_1}-\hat{u}_{\alpha_1}=c_{\alpha_1}-c_{\alpha_2}>0$ on
$\partial \hat{F}$. By applying the maximum principle to
(\ref{eqa23}), we deduce $u_{\alpha_1}-\hat{u}_{\alpha_1}\le
c_{\alpha_1}-c_{\alpha_2}$ in $\hat{F}$. Recalling the definition of
$\hat{F}$, we have $\hat{F}=\emptyset$. Since
$\tilde{D}_{\alpha_1}\setminus\hat{D}_{\alpha_1}\subseteq \hat{F}$, we
infer $\tilde{D}_{\alpha_1}\setminus\hat{D}_{\alpha_1}=\emptyset$,
{\ie}~$\tilde{D}_{\alpha_1}\subseteq\hat{D}_{\alpha_1}$. So, from
(\ref{eqa24}) and the fact that $|\tilde{D}_{\alpha_1}|=\gamma$, we
deduce $|\hat{D}_{\alpha_1}\setminus\tilde{D}_{\alpha_1}|$ decreases
to zero. Furthermore, recalling that
$|\hat{D}_{\alpha_1}\setminus\tilde{D}_{\alpha_2}|$ decreases to zero,
from (\ref{eqa20}) we have:
\begin{equation*}
|\tilde{D}_{\alpha_1}\vartriangle\tilde{D}_{\alpha_2}|=|(\tilde{D}_{\alpha_1}\setminus\tilde{D}_{\alpha_2})\cup
(\tilde{D}_{\alpha_2}\setminus\tilde{D}_{\alpha_1})|\le
|\hat{D}_{\alpha_1}\setminus\tilde{D}_{\alpha_2}|+
|\hat{D}_{\alpha_1}\setminus\tilde{D}_{\alpha_1}| \to 0^+,
\end{equation*}
when $\alpha_1$ increases to $\alpha_2$ as desired. Similarly, when
$\alpha_1$ decreases to $\alpha_2$, with $\beta<\alpha_2<1$, we will
have $|\tilde{D}_{\alpha_1}\vartriangle\tilde{D}_{\alpha_2}|$
converging to zero. This completes the proof.
\end{proof}

\begin{theorem}
\label{continuitya3} 
If $\beta>0$ and $\alpha_1$ converges to $\alpha_2$ in $(\beta,1]$,
then $u_{\alpha_1}$ converges to $u_{\alpha_2}$ in $C(\bar{D})$.
\end{theorem}

\begin{corollary}\label{continuitya4} If $\beta>0$ and $\alpha_1$ converges to $\alpha_2$ in $(\beta,1]$, then $\Psi(\alpha_1)$ converges to
$\Psi(\alpha_2)$.
\end{corollary}

\section{Numerical simulation}
\label{sec:numerical-simulation}

Numerical algorithms for solving rearrangement optimization problems
have appeared in the literature (See,
{\eg}, \cite{Elcrat:iter_Steady_Vortices_Rearrangement:1995,Emamizadeh_Farjudian_Zivari:Nonlocal_Kirchhoff:2016,Emamizadeh_Farjudian_Liu-Harvesting:2017}). As
there are no non-global local minima for problem
(\ref{eq:main_problem}), a simple gradient descent algorithm
suffices. Thus, we do not discuss the details of the algorithm here.

Nonetheless, we highlight a few issues regarding numerical simulation
of the problem (\ref{eq:main_problem}). It is clear from the
variational formulation (\ref{eq:variational_formulation}) that the
optimization problem (\ref{eq:main_problem}) is a minmax one. Speeding
up algorithms for rearrangement problems of this kind requires dealing
with certain heuristics, which are discussed in detail by Kao and
Su~\cite{Kao_Su:Efficient_Rearrangement_Algorithms:2013}.

The optimization problem (\ref{eq:main_problem}) of the current paper
should be contrasted with (say) the optimal harvesting problem
of \cite{Emamizadeh_Farjudian_Liu-Harvesting:2017}, or the steady
vortex problem considered
in \cite{Burton:1989,Burton:saddle:1989}. Here are two major
differences:

\begin{enumerate}[(1)]
\item Whereas the steady vortex and optimal harvesting problems can
  have uncountably many local optima and saddle points---which may
  only be \emph{partially} overcome through the use of randomized
  algorithms~\cite{Emamizadeh_Farjudian_Liu-Harvesting:2017}---problem
  (\ref{eq:main_problem}) has no non-global local minima.
\item On the other hand, the maxmax nature of the steady vortex
  problem and the minmin nature of the optimal harvesting problem
  provide for highly efficient algorithms that generate optimizing
  sequences. For problem (\ref{eq:main_problem}), however, careful use
  of heuristics is needed.
\end{enumerate}

Using an approach similar to that
of \cite{Kao_Su:Efficient_Rearrangement_Algorithms:2013}, we have
implemented an algorithm for the shape optimization problem
(\ref{eq:min_prob_alpha_beta}). Figure~\ref{fig:dumbbell} (generated
by MATLAB{\textsuperscript{\textregistered}}) illustrates one of our
monotonicity results, as stated in Proposition \ref{prop:monot_gamma}
(\ref{item:MRa2}).

\begin{figure}
  \centering
  \begin{subfigure}[b]{0.3\textwidth}
    \includegraphics[width=\textwidth]{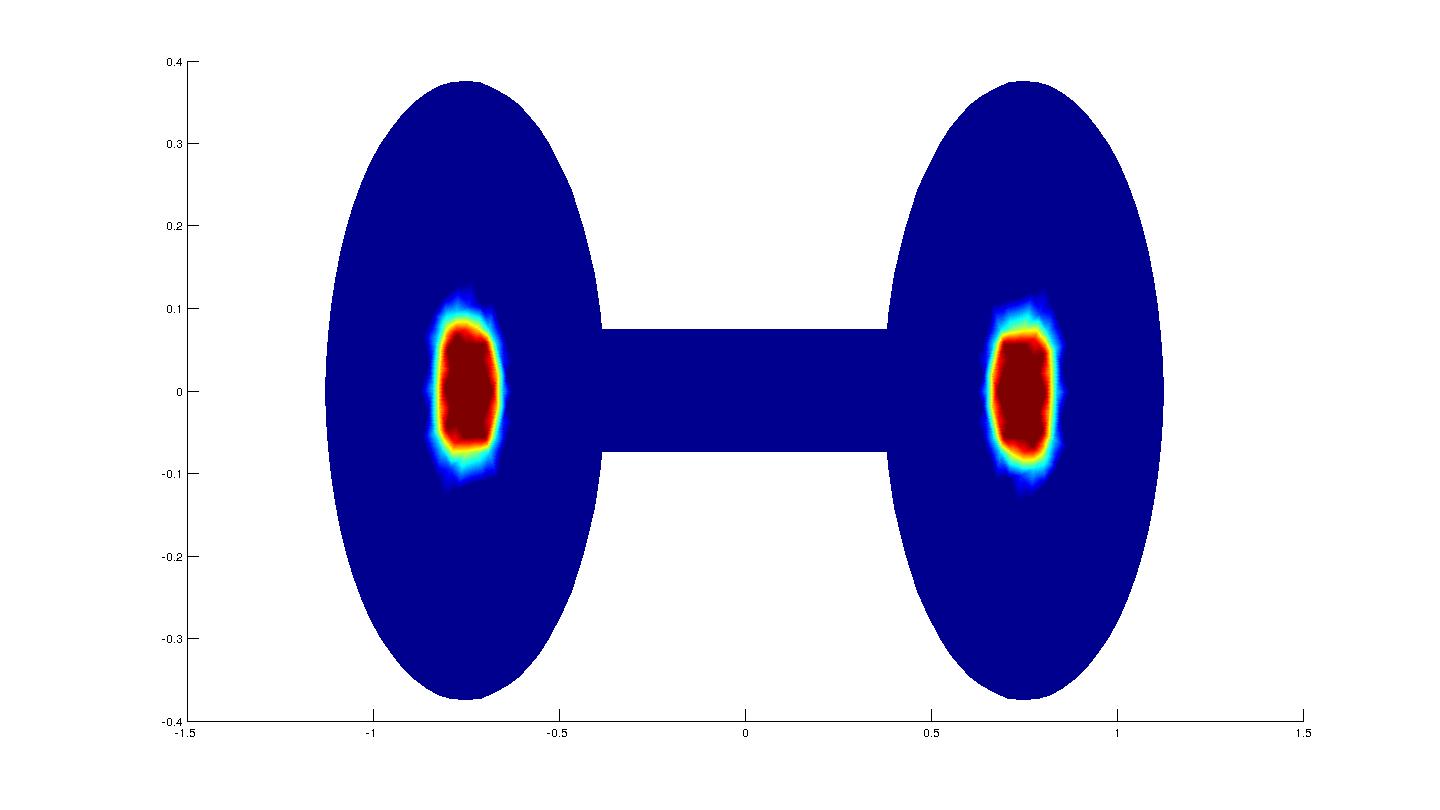}
    \caption{$\gamma = 0.05$}
    \label{fig:dumbbell_005_minimizer}
  \end{subfigure}
 ~
  \begin{subfigure}[b]{0.3\textwidth}
    \includegraphics[width=\textwidth]{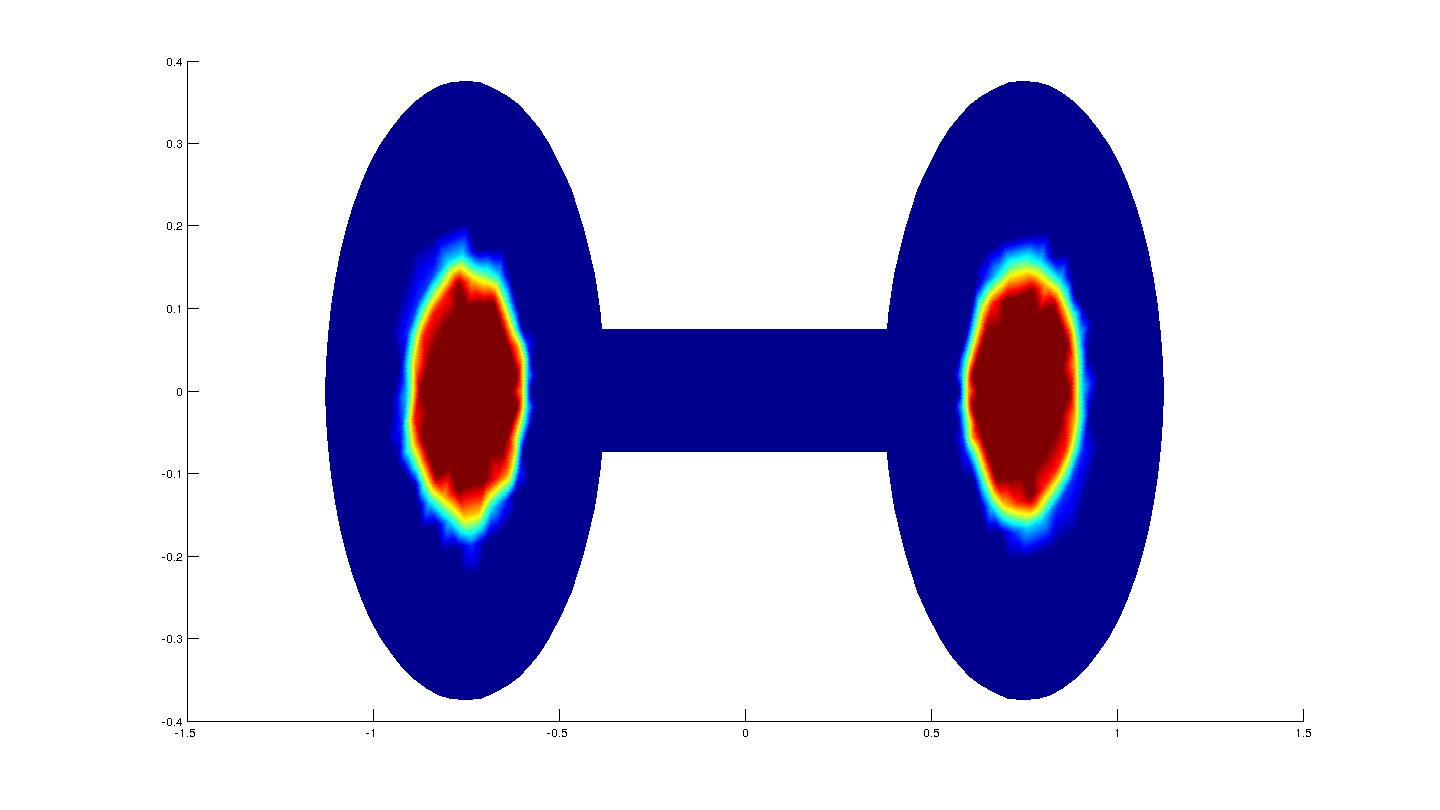}
                \caption{$\gamma = 0.15$}
                \label{fig:dumbbell_015_minimizer}
        \end{subfigure}%
        ~ 
        \begin{subfigure}[b]{0.3\textwidth}
                \includegraphics[width=\textwidth]{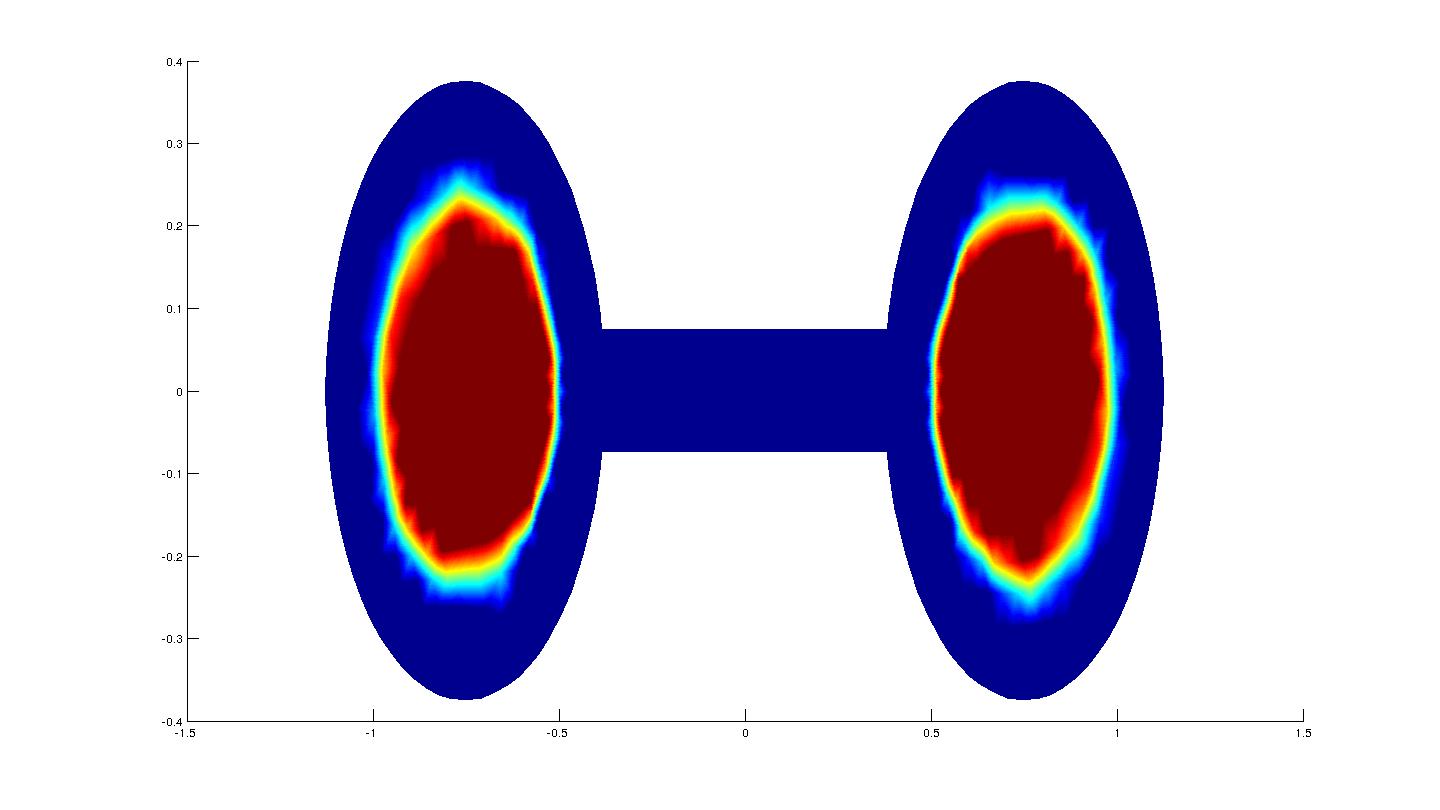}
                \caption{$\gamma = 0.35$}
                \label{fig:dumbbell_035_minimizer}
        \end{subfigure}
        \caption{Dumbbell shaped domain $D$ with $\absn{D} = 1$. It is
          clear that $\tilde{D}_{0.05} \subseteq \tilde{D}_{0.15}
          \subseteq \tilde{D}_{0.35}$, as stated in
          Proposition \ref{prop:monot_gamma} (\ref{item:MRa2}).}
        \label{fig:dumbbell}
\end{figure}

\section{Proof of Lemma~\ref{lem:basic_props_of_energy_func}}
\label{sec:proof-lemma-phi}

\begin{enumerate}[(i)]
\item We follow the ideas
  in~\cite{EmamizadehLiu:Constrained_unconstrained:2015} (also,
  see~\cite{Marras:2010:Optim_pLap}). Let $\left\{g_n\right\}
  \subseteq \overline{\cal{R}}$ and $g \in \overline{\cal{R}}$, such
  that $g_n\rightharpoonup g$ in $L^2(D)$. For simplicity, let us
  set $u_n \defeq u_{g_n}$ and $u \defeq u_g$. We have:
  \begin{equation}
    \label{eq:bvp2}
    \left\{
      \arrayoptions{1ex}{1.2}
      \begin{array}{ll}
        -\Delta u_n+g_nu_n=f &\text{in}\;D,\\
        u_n=0 &\text{on}\;\partial D.
      \end{array}
    \right.
  \end{equation}
  Multiplying the differential equation in (\ref{eq:bvp2}) by $u_n$, and
  integrating the result over $D$, yields
  \begin{equation}\label{eq3}\int_D|\nabla
    u_n|^2 \md x+\int_Dg_nu_n^2 \md x=\int_Dfu_n \md x.
  \end{equation}
  From Lemma \ref{lem:rearrange3}, we know that $g_n$ are
  non-negative. Therefore (\ref{eq3}) implies
  \begin{equation}\label{eq26}\int_D|\nabla u_n|^2 \md x\le\int_Dfu_n \md x.\end{equation} By applying H{\"o}lder's inequality and the
  Poincar$\acute{\text{e}}$ inequality to the right hand side of
  (\ref{eq26}) we obtain
  \begin{equation}
    \label{eq4}
    \int_D|\nabla u_n|^2 \md x\le C\left\|f\right\|_2\left\|
      u_n\right\|_{H_0^1(D)},
  \end{equation}
  in which $C$ is a positive constant. Whence, $\left\{u_n\right\}$ is a
  bounded sequence in $H_0^1(D)$. This in turn implies existence of a
  subsequence of $\{u_n\}$, still denoted $\{u_n\}$, and $w\in
  H^1_0(D)$, such that:
  \begin{equation*}
    u_n\rightharpoonup
    w\;\;\;\text{in}\;H^1_0(D)\;\;\;\;\;
    \text{and}\;\;\;\;\;u_n\rightarrow w\;\;\;\text{in}\;L^2(D).
  \end{equation*}
  Let us prove that $w=u$, where $u$ is the solution of
  \begin{equation}
    \label{bvp20}
    \left\{
      \arrayoptions{1ex}{1.2}
      \begin{array}{ll}
        -\Delta u+gu=f &\text{in}\;D,\\
        u=0 &\text{on}\;\partial D.
      \end{array}
    \right.
  \end{equation}
  Indeed, by (\ref{eq:bvp2}) we have
  $$\int_D\nabla u_n\cdot\nabla\phi \,  \md x+\int_Dg_nu_n\phi  \md x=\int_Df
  \phi  \md x,\ \ \ \ \forall \phi\in C^\infty_0(D).$$ Since
  $u_n\rightharpoonup w$ in $H^1_0(D)$, $g_n\rightharpoonup g$ in
  $L^2(D)$, and $u_n\to w$ strongly in $L^2(D)$, from the latter
  equation we find
  \begin{equation*}
    \int_D\nabla w\cdot\nabla\phi \,  \md x+\int_Dgw\, \phi  \md x=\int_Df
    \phi  \md x,\ \ \ \ \forall \phi\in C^\infty_0(D).
  \end{equation*}
  This means that $w$ is a solution of (\ref{bvp20}), and by uniqueness,
  we must have $w=u$.  To prove (\ref{item:weakly_continuous}), we
  observe that
  \begin{equation*}
    \Big|\Phi(g_n)-\Phi(g)\Big|=\Big|\int_Df(u_n-u) \md x\Big|\le
    ||f||_2 \,  ||u_n-u||_2
  \end{equation*}
  which together with the fact that $\lim_{n \to \infty}||u_n-u||_2 = 0$
  implies (\ref{item:weakly_continuous}).

\item Let $h,g\in\overline{\cal{R}}$, $0<t<1$, and
  $\xi_t=th+(1-t)g$. For $v\in H^1_0(D)$, we have
  \begin{multline}
    \label{eq6}
    2\int_Dfv \md x-\int_D|\nabla
    v|^2 \md x - \int_D\xi_tv^2 \md x  =  \\
    t\left(2\int_Dfv \md x-\int_D|\nabla v|^2 \md x -\int_Dhv^2 \md x\right) \\
    + (1-t)\left(2\int_Dfv \md x-\int_D|\nabla v|^2 \md x - \int_Dgv^2 \md x\right)
  \end{multline}
  By taking the supremum of (\ref{eq6}) with respect to $v\in
  H^1_0(D)$, we obtain
  \begin{equation}
    \label{eq7}
    \Phi(th+(1-t)g)\le t\Phi(h)+(1-t)\Phi(g).
  \end{equation}
  This proves the convexity of $\Phi$.  We now show, by
  contradiction, that $\Phi$ is in fact strictly convex. To this
  end, we assume that there exists $t\in(0,1)$ such that
  $\Phi(th+(1-t)g)=t \, \Phi(h)+(1-t)\Phi(g)$. For simplicity, we
  use $u_t$ in place of $u_{th+(1-t)g}$. So, we have:
  \begin{multline}
    \label{eq8}
    2\int_Dfu_t \md x-\int_D|\nabla
    u_t|^2 \md x - \int_D\xi_tu_t^2 \md x   = \\
    t\left(2\int_Dfu_h \md x-\int_D|\nabla u_h|^2 \md x -\int_Dhu_h^2 \md x\right)\\
    +  (1-t)\left(2\int_Dfu_g \md x-\int_D|\nabla u_g|^2 \md x -
      \int_Dgu_g^2 \md x\right).
  \end{multline}
  From (\ref{eq8}), we deduce the following equations:
  \begin{equation}
    \label{eq27}
    2\int_Dfu_h \md x-\int_D|\nabla u_h|^2 \md x -\int_Dhu_h^2 \md x =
    2\int_Dfu_t \md x-\int_D|\nabla u_t|^2 \md x-\int_Dhu_t^2 \md x,
  \end{equation}
  and
  \begin{equation}
    \label{eq28}
    2\int_Dfu_g \md x-\int_D|\nabla
    u_g|^2 \md x -\int_Dgu_g^2 \md x =
    2\int_Dfu_t \md x-\int_D|\nabla u_t|^2 \md x-\int_Dgu_t^2 \md x.
  \end{equation}
  From the maximality of $u_h$ coupled with (\ref{eq27}), we infer
  $u_h=u_t$. Similarly, from the maximality of $u_g$ and
  (\ref{eq28}), we find $u_g=u_t$.  Hence, $u_t=u_h=u_g$. On the
  other hand, from the differential equations $$-\Delta
  u_h+hu_h=f,\;\ \text{a.e. in}\,\ D,$$ and
  \begin{equation*}
    -\Delta u_g+gu_g=f,\;\ \text{a.e. in}\,\ D,
  \end{equation*}
  we infer $(h-g)u_h=0$, almost everywhere in $D$. Since $u_h$ is
  positive by the strong maximum principle, we must have $h=g$
  almost everywhere in $D$. Therefore, the strict convexity is
  proved.

\item For simplicity, we set $u_t \defeq u_{\xi_t}$. We know that:
  \begin{equation}
    \label{bvp10}
    \left\{
      \arrayoptions{1ex}{1.2}
      \begin{array}{ll}
        -\Delta u_t+\xi_tu_t=f &\text{in}\;D, \\
        u_t=0 &\text{on}\;\partial D,
      \end{array}
    \right.
  \end{equation}
  and
  \begin{equation}
    \label{bvp11}
    \left\{
      \arrayoptions{1ex}{1.2}
      \begin{array}{ll} -\Delta u+gu=f &\text{in}\;D,\\
        u=0 &\text{on}\;\partial D.
      \end{array}
    \right.
  \end{equation}
  From (\ref{bvp10}) and (\ref{bvp11}), we obtain:
  \begin{equation}
    \label{eq29}
    -\Delta (u_t-u)+g(u_t-u)=gu_t-\xi_tu_t=(g-\xi_t)u_t.
  \end{equation}
  Multiplying (\ref{eq29}) by $u_t+u$, and integrating the result
  over $D$, we get:
  \begin{multline}
    \label{eq30}
    \int_D|\nabla u_t|^2 \md x-\int_D|\nabla u|^2 \md x +\int_D gu_t^2 \md x -
    \int_Dgu^2 \md x \\
    = \int_D(g-\xi_t)u_t(u_t+u) \md x
    = -t\int_D(h-g)u_t(u_t+u) \md x.
  \end{multline}
  From (\ref{eq30}), we derive
  $\Phi(\xi_t)-\Phi(g)=-t\int_D(h-g)u_tu \md x$, which in turn implies:
  \arrayoptions{0.5ex}{1.1}
  \begin{eqnarray}
    \label{eq31}
    \Phi(\xi_t)-\Phi(g)+t\int_D(h-g)u^2 \md x&=&-t\int_D(h-g)u_tu \md x+t\int_D(h-g)u^2 \md x
    \nonumber \\ &=&-t\int_D(h-g)(u_t-u)u \md x.
  \end{eqnarray}
  By applying H{\"o}lder's inequality to the right hand side of (\ref{eq31}), we find
  \begin{equation}\label{eq32}\left|\Phi(\xi_t)-\Phi(g)+t\int_D(h-g)u^2 \md x\right|\le t\left\|h-g\right\|_\infty\left\|u_t-u\right\|_2\left\|u\right\|_2.
  \end{equation}
  Since $\xi_t\rightharpoonup g$ weakly in $L^2(D)$ (and even strongly),
  by the proof of part (\ref{item:weakly_continuous}), we have
  $||u_t-u||_2\to 0$ as $t\to 0$. Hence, dividing by $t$ in (\ref{eq32})
  and letting $t\to 0$ we get the desired result. \qed
\end{enumerate}


\section{Concluding remarks}

\label{sec:concluding-remarks}

In the main result of the current paper, {\ie},
Theorem~\ref{thm:main}, we proved existence and uniqueness of
solutions for an optimization problem arising in construction of
robust membranes, with no restriction on the number of materials
used. This is yet another witness to the power and elegance of the
theory behind optimization of convex functionals over rearrangement
classes, as laid out by Burton~\cite{Burton:1987}. Although the theory
was originally devised for studying vortex rings, {\ie}, in the
context of fluid dynamics, ever since its introduction, there has been
a steady flow of contribution to the theory and its applications, in
fluid
mechanics~\cite{Elcrat:Rearrangements_steady_vortex:1991,Elcrat:iter_Steady_Vortices_Rearrangement:1995},
finance~\cite{Emamizadeh_AlHanai:Real_Estate:2009,Rueschendorf:Risk_Analysis:Rearrangements:book:2013},
free boundary problems~\cite{EmamizadehMarras:RearrFreeBound:2014},
population biology~\cite{Emamizadeh_Farjudian_Liu-Harvesting:2017},
and eigenvalue problems \cite{EmamizadehZivari:Steklov:2011}, to name
a few.

For the particular problem considered in the current paper, we showed
that there cannot be any non-global local optima
(Theorem~\ref{thm:local_minim}). This has to be contrasted with other
rearrangement optimization problems where local optima and saddle
points
abound~\cite{Burton:1989,Burton:saddle:1989,Emamizadeh_Farjudian_Zivari:Nonlocal_Kirchhoff:2016,Emamizadeh_Farjudian_Liu-Harvesting:2017}. Furthermore,
we managed to deepen our understanding of the problem through some
stability results, which are, very difficult to prove, or even
formulate, in the presence of symmetry breaking, such as those
occurring
in~\cite{Emamizadeh_Farjudian_Zivari:Nonlocal_Kirchhoff:2016,Emamizadeh_Farjudian_Liu-Harvesting:2017}.




\bibliographystyle{plain}


\end{document}